\documentclass[letterpaper, 11pt]{article}

\usepackage{fullpage}
\usepackage{graphicx}
\usepackage{latexsym}
\usepackage{amsmath, amssymb, amsfonts, amsthm}
\usepackage{url}

\usepackage{sparklines}

\usepackage[pdftex,bookmarks,colorlinks,breaklinks]{hyperref}  
\usepackage{color}
\definecolor{dullmagenta}{rgb}{0.4,0,0.4}
\definecolor{darkblue}{rgb}{0,0,0.4}
\hypersetup{linkcolor=red,citecolor=blue,filecolor=dullmagenta,urlcolor=darkblue} 

\newcommand{\reg}{\delta}

\numberwithin{equation}{section}
\newcommand{\nc}{\newcommand}

\nc{\fig}[4]
{
  \begin{figure}[ht!]  
    \centering{\scalebox{#1}{\includegraphics*{./figures/#2.eps}}}
    \caption{#4}
    \label{fig:#3}
  \end{figure}
}

\nc{\bXYZ}{{\bf XYZ\ }}

\nc{\nn}{\nonumber} 
\nc{\nit}{\noindent}
\nc{\marginnote}[1] {\marginpar{\tiny #1}}
\nc{\SH}{Schr\"odinger}
\nc{\NLS}{nonlinear Schr\"odinger}
\nc{\ie}{\emph{i.e.\ \mbox{}}}
\nc{\eg}{\emph{e.g.\ \mbox{}}}
\nc{\per}{{l_j}}
\nc{\perx}{q}
\nc{\vol}{{\rm Vol}}

\nc{\D}{\partial} 
\nc{\lapx}{\dfrac{\partial^2}{\partial x^2}}
\nc{\lapy}{\dfrac{\partial^2}{\partial y^2}}
\nc{\lapxy}{\dfrac{\partial^2}{\partial xy}}
\nc{\diff}[2]{\frac{d #1}{d #2}}
\nc{\diffn}[3]{\frac{d^{#3} #1}{d {#2}^{#3}}} 
\nc{\pdiff}[2]{\frac{\partial #1}{\partial #2}} 
\nc{\pdiffn}[3]{\frac{\partial^{#3} #1}{\partial{#2}^{#3}}} 

\def\Xint#1{\mathchoice
  {\XXint\displaystyle\textstyle{#1}}%
  {\XXint\textstyle\scriptstyle{#1}}%
  {\XXint\scriptstyle\scriptscriptstyle{#1}}%
  {\XXint\scriptscriptstyle\scriptscriptstyle{#1}}%
  \!\int}
\def\XXint#1#2#3{{\setbox0=\hbox{$#1{#2#3}{\int}$} \vcenter{\hbox{$#2#3$}}\kern-.5\wd0}}
\def\dashint{\Xint-}
\nc{\Av}{\dashint_\cell}
\nc{\avg}[1]{\mbox{$\left \langle\ \!#1\!\!\ \right \rangle$}}
\nc{\intRd}{\int_{{\mathbb R}^d}}

\nc{\abs}[1] {\lvert #1 \rvert} 
\nc{\norm}[2] {{\lVert #1 \rVert}_{#2}} 
\nc{\normH}[1]{\|#1\|_{H^1}^2}
\nc{\normL}[1]{\|#1\|_2^2}
\nc{\nl}[1]{|#1|^{2\sigma}}
\nc{\nlb}[1]{#1^{2\sigma+1}}
\nc{\Linvd}{L_\delta^{-1}}
\nc{\Linvz}{L_0^{-1}}
\nc{\Linvzs}{L_0^{-2}}

\DeclareMathOperator{\sech}{sech}

\nc{\st}[1]{\stackrel{(\ref{eq:#1})}{=}}
\nc{\stt}[2]{\stackrel{(\ref{eq:#1}),(\ref{eq:#2})}{=}}

\nc{\eps}{\epsilon}
\nc{\veps}{\epsilon}
\nc{\Ue}{U_\eps}
\nc{\koe}{\frac{k}{\veps}} 

\nc{\wrat}{w_{\rm ratio}}
\nc{\aij}{A^{ij}}                  
\nc{\minv}{m_*^{-1}} 
\nc{\sqminv}{m_*^{-\frac{1}{2}}}
\nc{\G}{\gamma_{\rm ef\,\!f}}
\nc{\zetap}{{\zeta_*}}
\nc{\zetas}{{\zeta_{1*}}}
\nc{\Pe}{P_{\rm edge}}
\nc{\slope}{\dfrac{d\cP[u(\cdot;\mu)] }{d\mu}}
\nc{\moot}{\mu_2}
\nc{\mum}{\mu_{\rm min}}

\def\R{\mathbb{R}}

\nc{\bk}{{\bf k}}
\nc{\bq}{{\bf q}}
\nc{\br}{{\bf r}}
\nc{\bx}{{\bf x}}
\nc{\by}{{\bf y}}
\nc{\bz}{{\bf z}}
\nc{\bl}{{\bf l}}
\nc{\bnu}{{\bf \nu}}
\nc{\bxc}{{\bf x}_c}
\nc{\bxi}{{\bold\xi}}

\nc{\cP}{{\cal P}}
\nc{\cPedge}{{\cal P}_{edge}}
\nc{\cPc}{{\cal P}_{cr}} 
\nc{\cell}{{\cal B}}
\nc{\cDD}{\Lambda} 
\nc{\cA}{{\cal A}}
\nc{\cB}{{\cal B}}
\nc{\cF}{{\cal F}}
\nc{\cG}{{\cal G}}
\nc{\cO}{{\cal O}}  
\nc{\cQ}{{\cal Q}}  
\nc{\cR}{{\cal R}}
\nc{\cI}{{\cal I}}
\nc{\cK}{{\cal K}}
\nc{\cL}{{\cal L}} 
\nc{\cM}{{\cal M}}
\nc{\cN}{{\cal N}} 
\nc{\cE}{{\cal E}}
\nc{\cH}{{\cal H}} 
\nc{\cX}{{\cal X}}
\nc{\cZ}{{\cal Z}} 
\nc{\cT}{{\cal T}} 
\nc{\order}{{\cal O}}

\newtheorem{thm}{Theorem}[section]
\newtheorem{cor}[thm]{Corollary}
\newtheorem{prop}[thm]{Proposition}
\newtheorem{lem}[thm]{Lemma}
\theoremstyle{definition}
\newtheorem{defn}{Definition}[section]
\theoremstyle{remark}
\newtheorem{rem}[thm]{\bf Remark}

\newtheorem{conj}{\bf Conjecture}[section]
\newtheorem{remark}{Remark}[section]
\newcommand{\no}{\noindent}
\newcommand{\ud}{\,\mathrm{d}}

\graphicspath{{./figs/}} 

\title{Emergence of Periodic Structure from
Maximizing the Lifetime of a Bound State 
Coupled to Radiation}
\author{Braxton Osting and  Michael I. Weinstein}
\date{\today}
 
\begin{document}
\maketitle

\no {\bf Keywords:} Fermi's golden rule, quality factor, constrained optimization, ionization, parametric forcing, Schr\"odinger Equation, Bragg resonance, spectral gap

\begin{abstract}
\noindent Consider a system governed by the time-dependent Schr\"odinger equation in its ground state. 
When subjected to weak  (size $\epsilon$) parametric forcing by an ``ionizing field'' (time-varying), the state decays with advancing time due to coupling of the bound state to radiation modes. The decay-rate of this metastable state is governed by {\it  Fermi's Golden Rule}, $\Gamma[V]$, which depends on the potential $V$ and the details of the forcing. We pose the potential design problem: find $V_{opt}$ which minimizes $\Gamma[V]$  (maximizes the lifetime of the state) over an admissible class of potentials with fixed spatial support. We formulate this  problem  as a constrained optimization problem and prove that an admissible optimal solution exists. Then, using quasi-Newton methods, we compute locally optimal potentials. These have the structure of a truncated periodic potential with a localized defect. In contrast to optimal structures for other spectral optimization problems, 
 our optimizing potentials appear to be interior points of the constraint set and to be smooth. The  multi-scale structures that emerge incorporate the physical mechanisms of energy confinement via material contrast and interference effects. 
 
 An analysis of locally optimal potentials reveals local optimality is attained via two mechanisms: (i) decreasing the density of states near a resonant frequency in the continuum and (ii) tuning the oscillations of extended states to make $\Gamma[V]$, an oscillatory integral, small. 
  Our approach achieves lifetimes, $\sim \left(\epsilon^2\Gamma[V]\right)^{-1}$, for locally optimal potentials with 
 $\Gamma^{-1}\sim\mathcal{O}(10^{9})$ as compared with  $\Gamma^{-1}\sim \mathcal{O}(10^{2})$ for a typical potential. 
    Finally, we explore the performance of optimal potentials via simulations of the time-evolution.
\end{abstract}


\section{Introduction}

In many problems of fundamental and applied science a scenario arises where, due to physical laws or engineering design,  a state of the system is {\it metastable}; the state is long-lived but of finite lifetime due to coupling or leakage to an {\it environment}. In settings as diverse as  linear and nonlinear optics, cavity QED, Bose-Einstein condensation (BEC), and quantum computation, one is interested in the manipulation of the lifetime of such metastable states.  Our goal in this paper is to explore the problem of maximizing the lifetime of a metastable state for a class of {\it ionization problems}. The approach we take is applicable to a wide variety of linear and nonlinear problems. 
Specific examples where metastable states exist include:
\begin{enumerate}
\item the excited state of an atom, {\it e.g.} hydrogen, where due to coupling to a photon field, the atom in its excited state spontaneously undergoes a transition to its ground state, after some excited state {\it lifetime} \cite{CohenTannoudji}, 
\item an approximate bound state (quasi-mode) of  a quantum system, {\it e.g.} atom
 in a cavity or BEC in a magnetic trap, which  leaks (``tunnels'') out of the cavity and whose wave function decays
 with advancing time \cite{PitaevskiiStringari}, 
\item an approximate guided mode of an optical waveguide which, due to scattering, bends in the waveguide, or diffraction, leaks out of the structure, resulting in  attenuation of the wave-field within the waveguide with increasing propagation distance \cite{Marcuse},  and 
\item   a ``scatterer'' which confines ``rays'' , but leaks energy to spatial infinity due to their wave nature, {\it e.g.} Helmholtz resonator, traps rays between obstacles \cite{LaxPhillips}. 
\end{enumerate}

These examples are representative of a class of extended (infinite spatial domain), yet  energy-preserving ({\it closed}) systems, where the mechanism for energy loss is {\it scattering loss}, the  escape of  energy from a compact spatial domain to spatial infinity.
  
    Such systems can often be viewed as  two coupled subsystems, one with oscillator-like degrees of freedom characterized by discrete  frequencies and the other  a wave-field  characterized by continuous spectrum. When (artificially) decoupled from the wave-field, the discrete system has infinitely long-lived time-periodic bound states. Coupling leads to energy transfer from the system with oscillator like degrees of freedom to the wave field. In many situations,  a (typically approximate) reduced description, which is a  
   closed equation for the oscillator amplitudes, can be derived.
   This reduction captures the view of the oscillator degrees of freedom as an {\it open} system  with an effective (radiation) damping term. In the problems considered in this paper, the reduced equation is of the simple form:
  \begin{equation}
  \imath \D_t A^\epsilon(t)\ \sim\ \epsilon^2\ \left(\ \Lambda-\  \imath \Gamma\ \right) A^\epsilon(t).
   \label{toy-oscillator}
   \end{equation}
   Here, $\epsilon$ is a real-valued small parameter, measuring the degree of coupling between oscillator and field degrees of freedom. $A^\epsilon(t)$  denotes the slowly-varying complex envelope amplitude of the perturbed bound state.  $\Lambda$ is a real frequency and $\Gamma>0$ is an effective damping, governing the rate of  transfer of energy from the oscillator to field degrees  of freedom.
  
 For example, consider the general linear or nonlinear Schr\"odinger equation
 \begin{equation}
 \imath \D_t\phi^\epsilon \, = \,  H_V\phi^\epsilon + \epsilon\ W(t,x,|\phi^\epsilon|) \ \phi^\epsilon, \quad \quad   H_V\equiv  -\Delta + V(x).\ 
 \label{general-nls}
 \end{equation}
 Here, $V(x)$ is a real-valued  time-independent potential and  $W(t,x,|\phi|)$ is a time-dependent
 potential (parametric forcing), $W=\tilde \beta(t,x)$, or nonlinear potential, {\it e.g.} $W= \pm  |\phi|^2$. 
  Equation \eqref{general-nls} defines an evolution which is unitary in $L^2(\mathbb{R})$.

 In this article we focus on the class of one-dimensional \emph{ionization} problems, where
     \begin{equation*}
    W(t,x)\ =\ \cos(\mu t)\ \beta(x)\ 
    \end{equation*}
    where $\beta(x)$ is a spatially localized and real-valued function and $\mu>0$
     is a parametric forcing frequency.
     Thus, our equation is a parametrically forced Schr\"odinger equation:
     \begin{equation}
     \imath \D_t\phi^\epsilon\ =\ H_V\phi^\epsilon\ +\ \epsilon\ \cos(\mu t)\ \beta(x)\ \phi^\epsilon.
     \label{forced-schrod}
     \end{equation} 
     We focus on the case where the parameter $\epsilon$ is real-valued and assumed sufficiently small. 
     \bigskip

\nit{\bf Assumptions for the unperturbed problem, $\epsilon=0$:}\ 
Initially we assume that the potential $V(x)$, decays sufficiently rapidly as $|x|\to\infty$,
although we shall later restrict to potentials with a fixed compact support. 
Furthermore, we assume that $H_V$ has exactly one eigenvalue $\lambda_V<0$, with corresponding (bound state)eigenfunction, $\psi_V(x)$:
 \begin{equation}
 H_V\ \psi_V\ =\ \lambda_V\ \psi_V, \quad  \|\psi_V\|_2=1.
 \label{psi0lam0}
 \end{equation}
 Thus, $\phi^0(x,t)=e^{-\imath \lambda_V t}\psi_V(x)$ is a time-periodic and spatially localized solution of the unperturbed linear Schr\"odinger equation:
  \begin{equation*}
 \imath \D_t\phi\ =\ H_V\phi 
 \end{equation*}
 We indicate an explicit dependence of  $\lambda_V$ and $\psi_V$ on $V$, since we shall be varying $V$. 
 
 \medskip \nit{\bf Fermi's Golden Rule:}\ 
  We cite consequences of the general theory of \cite{Soffer-Weinstein:98,Kirr-Weinstein:01,Kirr-Weinstein:03}. 
       If $\mu$, the forcing frequency,  is such that $\lambda_V+\mu>0$, then for initial data, $\phi(x,0)=\psi_V(x)$ 
        (or close to $\psi_V$), 
    the solution decays to zero as $t\to\infty$. On a time scale of order $\epsilon^{-2}$ 
    the decay is controlled by \eqref{toy-oscillator}, {\it i.e.}
  \begin{equation}
  |A(t)| \, \sim \, |A(0)| \,   e^{-\epsilon^2\Gamma[V] t},\ \ 0< t < \cO(\epsilon^{-2})\,
  \label{toy-decay}
  \end{equation}
  where $|A(t)|=\left|\langle \psi_V,\phi^\epsilon(t)\rangle\right|$ and $\Gamma[V]$ is a positive constant.
 Thus, we say the bound state has a lifetime of order $\left(\epsilon^2\cdot \Gamma[V]\right)^{-1}$ and the perturbation {\it ionizes} the bound state.

 The emergent damping coefficient, $\Gamma[V]$, is often called {\it Fermi's Golden Rule}  \cite{Visser:2009le}, arising in the context of the spontaneous emission problem. However, the notion of effective radiation damping due to coupling of an oscillator to a field has a long history  \cite{Lamb:00}. In general, $\Gamma\left[V\right]$ is a sum of expressions of the form:

 \begin{equation}
 \left|\ \langle e_V(\cdot,k_{res}(\lambda_V)),{\cal G}_W(\psi_V) \rangle_{L^2(\R^d)} \right|^2
  = |t_V(k_{res})|^2\  \left|\ \langle f_V(\cdot,k_{res}(\lambda_V)),{\cal G}_W(\psi_V) \rangle_{L^2(\R^d)} \right|^2,
 \label{toy-fgr}
 \end{equation}
 (see \eqref{eqn:Gamma}) where ${\cal G}_W(\psi_V)$ depends on the coupling perturbation $W$ in \eqref{general-nls}  and the unperturbed bound state, $\psi_V$. 
 Here, $e_V(\cdot,k_{res})=t_V(k_{res}) f_V(\cdot,k_{res})$ is the {\it distorted plane wave} (continuum radiation mode) associated with the Schr\"odinger operator, $H_V$, at a {\it resonant frequency} $k_{res} = k_{res}(\lambda_V)$, for which $k_{res}^2\in \sigma_{cont}(H_V)$. $t_V(k)$ denotes the transmission coefficient and $f_V(x,k)$ a {\it Jost solution}. 
  In Secs. \ref{sec:Scattering} and \ref{sec:FGR-ionize} we present an outline of the background theory for scattering and the ionization problem, leading to \eqref{toy-decay}, \eqref{toy-fgr}; see \cite{Soffer-Weinstein:98}. 
  \bigskip
 
 We study the problem of maximizing the lifetime of a metastable state, or equivalently, minimizing the scattering loss of a state due to radiation by appropriate deformation of the potential, $V(x)$, within some admissible class, $\cA_1(a,b,\mu)$:
 \begin{equation}
 \min_{V\in\cA_1(a,b,\mu)} \Gamma\left[V\right]. 
 \label{toy-opt}
 \end{equation}
 We refer to Eq. (\ref{toy-opt}) as the potential design problem (PDP). 
Our admissible class, $\cA_1(a,b,\mu)$, is defined as follows:

\begin{defn}
\label{def:VAdmiss}
$V\in \cA_1(a,b,\mu)$ if 
\begin{enumerate}
\item $V$ has support contained in the interval $[-a,a]$, {\it i.e.} $V\equiv0$ for $|x|>a$
\item $V\in H^1(\mathbb R)$ and $\|V\|_{H^1} \leq b$  
\item $H_V=\ -\D_x^2+V(x)$ has exactly one negative eigenvalue, $\lambda_V$,
with corresponding eigenfunction  $\psi_V\in L^2(\R)$, which satisfies Eq. (\ref{psi0lam0})\ : $H_V\psi_V=\lambda_V\psi_V,\ \ \|\psi_V\|_2=1$.
\item $k_{res} \equiv \sqrt{\lambda_V+\mu}>0$ (formal coupling to continuous spectrum)
\end{enumerate}
\end{defn}

\begin{rem}
Based on our numerical simulations, we conjecture that the hypothesis 2., imposing a bound of $V$,
  can be dropped. 
\end{rem}
\bigskip

The idea of controlling the lifetime of states by varying the characteristics of a background potential goes back to the work of E. Purcell \cite{Purcell:1946sf,Purcell1952}, who reasoned that the lifetime of a state can be influenced by manipulating the set of states to which it can couple, and through which it can radiate.

\begin{rem}\label{rem:2classes}
We discuss the potential design problem where
\begin{enumerate}
\item $\beta(x)$ is a fixed function, chosen independently of $V$, for example, $\beta(x) = \mathbf{1}_{[-2,2]}(x)$
\item $\beta(x) = V(x)$. 
\end{enumerate}
\end{rem}
\bigskip

\begin{rem} 
\label{rem:minGamma}
{\bf How does one minimize an expression of the form \eqref{toy-fgr}?}\\
 We can think of two ways in which \eqref{toy-fgr} can be made small:\medskip

 \noindent {\bf Mechanism (A)} Find a potential in $\mathcal{A}_1$ for which the first factor in \eqref{toy-fgr}, $|t_V(k_{res})|^2$ is small, corresponding to low {\it density of states} near $k_{res}^2$. \\
 {\bf N.B.} As proved in Proposition 
\ref{prop:tkGreaterZero}, $|t_V(k)|\ge \mathcal{O}(e^{-Ka})$ for $V$ with support contained in $[-a,a]$.\\
 
 \noindent {\bf Mechanism (B)} Find a potential in $\mathcal{A}_1$ which may have significant density of states near $k_{res}^2$  (say $|t_V(k_{res})|\ge 1/2$) but such that the oscillations of $f_V(x,k_{res})$ are tuned to make the matrix element expression (inner product) in \eqref{toy-fgr} small due to cancellation in the integral.
\medskip

\noindent Indeed, we find that both mechanisms occur in our optimization study. 
\end{rem}

 \medskip
 
 \begin{rem}\label{rem:contrast-interf}
 We are interested in the problem of deforming $V$ within an admissible set in such a way as to maximize the lifetime 
  of decaying (metastable) state. 
Intuitively, there are two physical mechanisms with which one can confine wave-energy in a region:
via the depth of the potential (material contrast) and via interference effects. We shall see that our (locally) optimal solutions, of types (A) and (B)  find the  proper balance of these mechanisms.
\end{rem}

\begin{figure}[t!]
\centering
\begin{tabular}{ccc}
\includegraphics[width=2in]{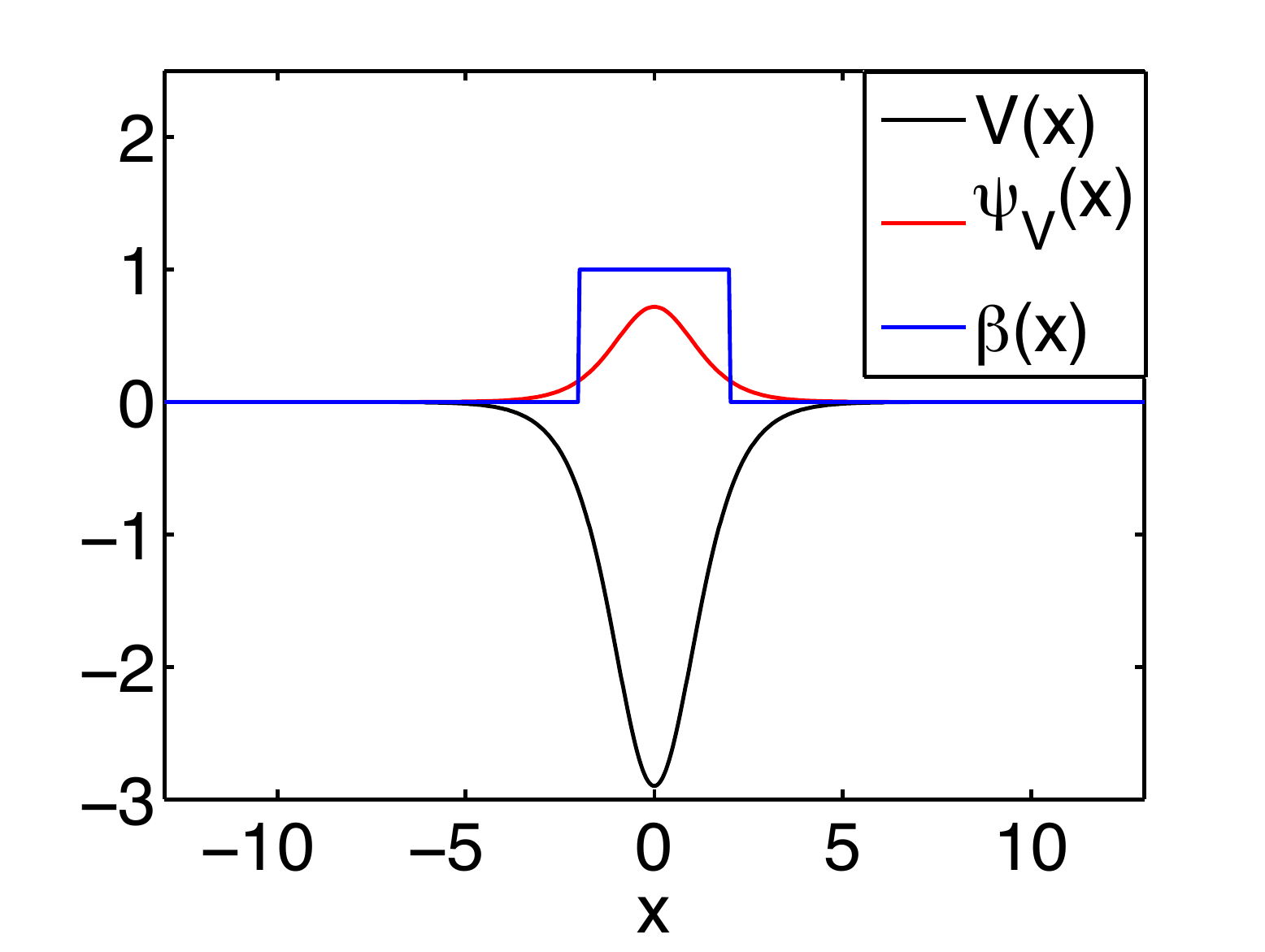}& 
\includegraphics[width=2in]{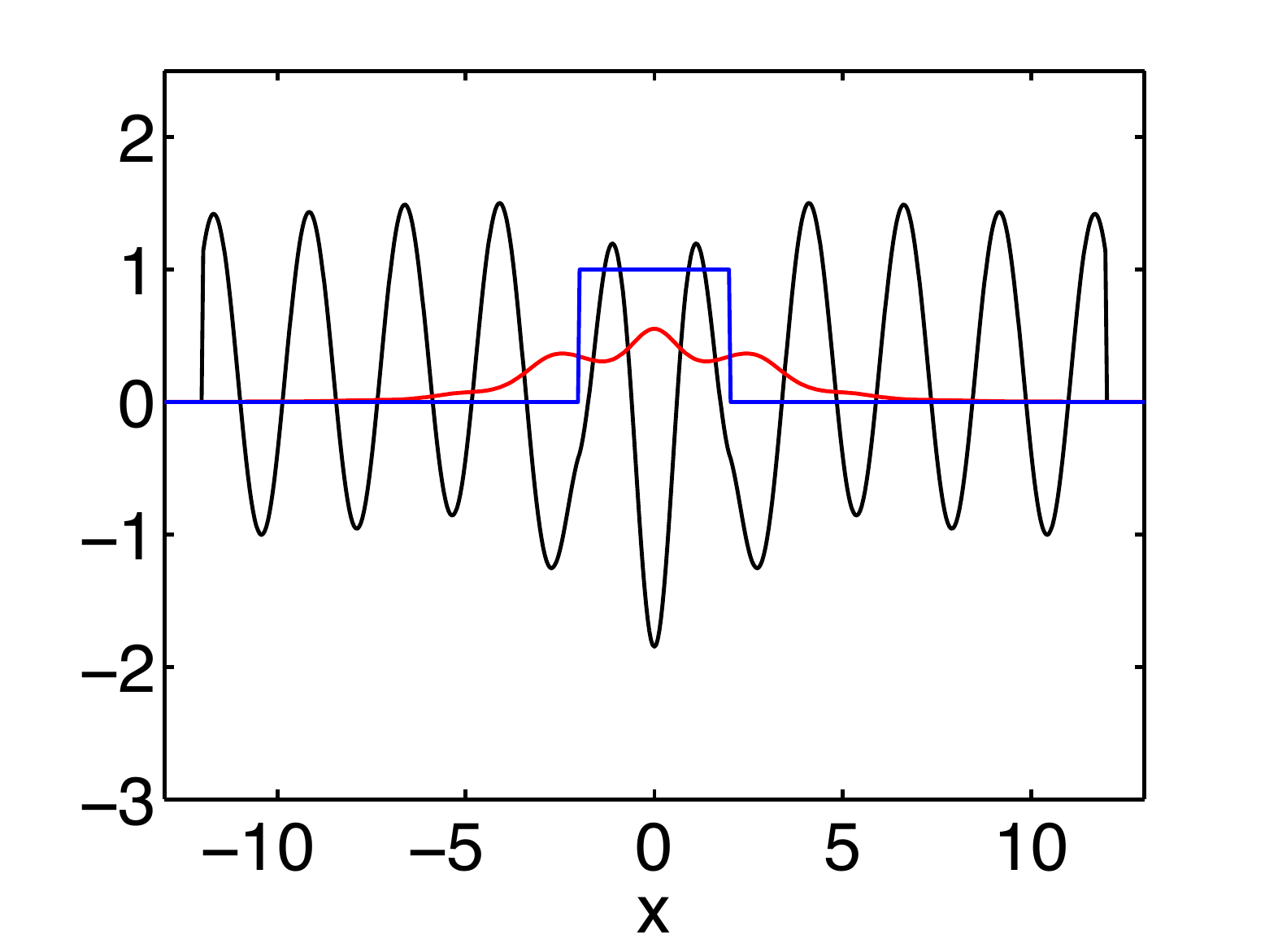} &
\includegraphics[width=2in]{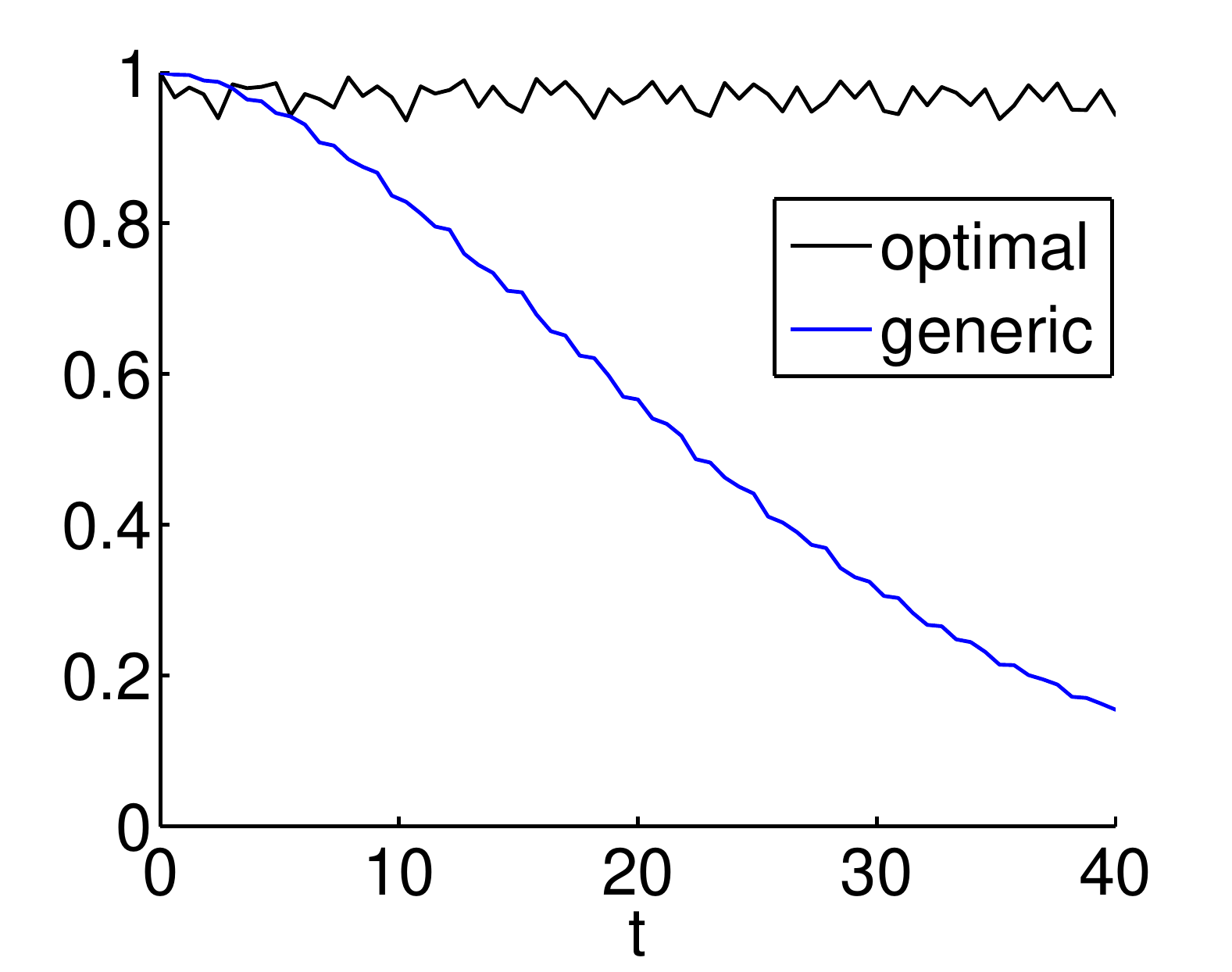}
\end{tabular}
\caption{Numerical demonstration for Eq. \eqref{forced-schrod} ($\epsilon=1$) of bound state time-decay for a typical potential well (left) and bound state persistence  for an optimized potential (center). We plot the potentials $V_{init}$ and $V_{opt}$ (black), corresponding ground states $\psi_{init}$ and $\psi_{opt}$ (red), and forcing function $\beta(x) = \mathbf{1}_{[-2,2]}(x)$ (blue). The rightmost figure displays the time evolution of the projection $|\langle \phi(t,\cdot),\psi_V{(\cdot)} \rangle |^2$ for each potential. Details are given in Sec. \ref{timeDependentSim}. $\Gamma[V_{init}] = 2.1\times10^{-2}$ and $\Gamma[V_{opt}] = 3.3\times10^{-9}$.}
\label{fig:nonAutOpt}
\end{figure}

\subsection{Overview of results:}
\begin{enumerate}
\item In Proposition \ref{thm:existence} we show that the optimal solution to Eq. (\ref{toy-opt}) exists, for an admissible set, $\cA_1^\reg(a,b,\mu)$, derived from $\cA_1(a,b,\mu)$ by  relaxing a discrete constraint; see Eqs. (\ref{eqn:Aeps}) and  (\ref{eqn:PDP}). 
\item Fix the admissible set $\cA^\reg_1(a,b,\mu)$, {\it i.e.} parameters $a,b,\mu,\reg$. 
Select an initial potential, $V_{init}\in\cA^\reg_1(a,b,\mu)$. 
For example, we have chosen a potential of the form $V_{init} = -A\sech(Bx)\ \mathbf{1}_{|x|\le a}$ with the parameters  $A$ and $B$ appropriately chosen.
We use a quasi-Newton method within  $\cA^\reg_1(a,b,\mu)$ and, after  about 50-100 iterations,  find a potential $V_{opt}\in \cA^\reg_1(a,b,\mu)$, for which $\Gamma[V]$ achieves a local minimum in $\cA^\reg_1(a,b,\mu)$.
\item In a typical search $\Gamma[V_{init}]\sim 10^{-2}$ and $\Gamma[V_{opt}]\sim 10^{-9}$. Therefore, by Theorem 
 \ref{thm:SW} \cite{Soffer-Weinstein:98,Kirr-Weinstein:01,Kirr-Weinstein:03}, the  decay time for the solution of \eqref{forced-schrod} with potential $V=V_{opt}$ and   data $\phi^\epsilon(0)= \psi_{V_{opt}}$ is much, much longer than that for the Schr\"odinger equation with potential $V=V_{init}$ and data $\phi^\epsilon(0)= \psi_{V_{init}}$. Thus our optimization procedure finds a potential which supports a metastable state which has a very long lifetime, in the presence of  parametric forcing coupling  to scattering states.
\item  As an independent check on the performance of our optimal structures, we numerically solve the initial value problem  for the time-dependent Schr\"odinger equation \eqref{forced-schrod} with $\epsilon =1$ for both $V=V_{init}$ with data $\phi^\epsilon(0)= \psi_{V_{init}}$ and $V=V_{opt}$ with data $\phi^\epsilon(0)= \psi_{V_{opt}}$. Figure \ref{fig:nonAutOpt} displays a representative comparison of these numerical experiments, revealing the decay of the bound state for $V_{init}$ and a striking persistence (non-decay) of the bound state for $V_{opt}$. The details of this simulation are given in Sec. \ref{timeDependentSim}. 
\item In Section \ref{sec:increaseSupport}, we investigate the optimization for classes of potentials with increasing support, {\it i.e.} $\cA_1^\reg(a,b,\mu)$ for an increasing sequence of $a$---values: $0<a_1<a_2<\dots<a_m$. Figure \ref{fig:compareParams} shows local optima found in $\cA^\reg_1(a_j,b,\mu)$.  As $a$ is taken larger,  the sequence $V_{opt,a_1}, V_{opt,a_2}, \dots V_{opt,a_m}$ appears to take on the character of a truncation to the interval $[-a,a]$ of a periodic structure with a localized defect. This suggests the following
\begin{conj}
\label{conj:convPer}
$\{ V_{opt,a} \}$ converges to $V_{opt,\infty}(x)=V_{per}(x)+ V_{loc}(x)$, where $V_{per}$ is periodic on $\R$ and $V_{loc}(x)$  is spatially localized.
\end{conj}
\item Our computational methods find locally optimal solutions which have small values of $\Gamma$ due to either of the mechanisms discussed in Remark  \ref{rem:minGamma} above. In Sec. \ref{sec:2mechs}, we find the confinement properties of potentials, which are optimal  due to the cancellation mechanism (\ mechanism (B)\ ), are very sensitive to perturbations in the forcing frequency away from the forcing frequency, $\mu$, for which the optimization is carried out. 
\item In section \ref{timeDependentSim} we study the stability or robustness of the  state, $\psi_{V_{opt}}$, for a locally optimal potential, $V_{opt}$. Time-dependent simulations  of the parametrically forced Schr\"odinger equation are performed for an un-optimized potential, $V_{init}$, and initial data $\psi_{V_{init}} +\ noise$ and  for $V_{opt}$, and initial data $\psi_{V_{opt}} +\ noise$. Optimal structures  effectively filter noise from a ground state, while a generic potential does not. The time scale of bound state radiation damping $\sim\left(\epsilon^2\Gamma[V_{opt}]\right)^{-1}$ is $\gg$ the time scale for dispersion of noise.
\item Our computations show that the inequality constraints of the (relaxed) admissible set, \eqref{eqn:Aeps},  are not active at   optimal potentials. This is in contrast to studies of other spectral optimization problems, {\it i.e.} scattering resonances   
\cite{II:1986uq,OptimizationOfScatteringResonances} and  band gaps \cite{Kao:2005pt,Sigmund:08}
 and other problems \cite{Krein:1955ye,Cox1990,osher-santosa-jcp01,Lipson-etal:06}, where periodic patterns attaining material bounds are obtained.
\end{enumerate}

\subsection{Outline of the article}

In section \ref{sec:Scattering} we introduce  the needed scattering theory background to explain 
 resonant radiative time-decay and Fermi's golden rule, $\Gamma[V]$, which characterizes the lifetime of metastable states. In section \ref{sec:FGR-ionize} 
 theory we summarize the theory of \cite{Soffer-Weinstein:98,Kirr-Weinstein:01,Kirr-Weinstein:03} in the context of the ionization problem \eqref{forced-schrod}. 
  In section \ref{sec:PDP} we introduce an appropriate regularization, $\mathcal{A}_1^\delta$,  of the admissible set of potentials, $\mathcal{A}_1$ (see Definition \ref{def:VAdmiss}) which is advantageous for numerical computation,  and prove the existence of a minimizer within this class. In section 
  \ref{sec:numMeth} we outline the numerical methods used to solve the optimization problem.
   In section \ref{sec:Results} we present numerical results for optimal structures and, as an independent check,
   investigate the effectiveness and robustness  of these structures for the time-dependent Schr\"odinger equation with optimized and un-optimized potentials. Section \ref{sec:Discussion} contains discussion and conclusions and Appendix  \ref{sec:appCompGrads} contains the detailed computations of functional derivatives and  gradients used in the optimization.

\subsection{Notation and conventions}
\begin{enumerate}
\item  $L^2(\mathbb R)$ inner product:  $\langle f,g\rangle = \int_{\mathbb R} \overline{f(x)} g(x) \ud x$
\item $L^2_{\text{comp}}(\mathbb R)$ is the space of compactly supported $L^2(\mathbb R)$ functions and $L^2_{\text{loc}}(\mathbb R)$ is the space of functions which are square-integrable on any compact subset of $\mathbb R$. 
\item Weighted $L^2$ space: 
$$
L^{2,s}(\mathbb R) = \{f\colon (1+|x|^2)^{\frac{s}{2}} f \in L^2( \mathbb R) \}, \quad \quad s \in \mathbb R
$$
with norm $\| f \|_{L^{2,s}}^2 = \int_{\mathbb R}  (1+|x|^2)^s f^2 \ud x$
\item Weighted Sobolev space:
$$
H^{k,s}(\mathbb R) = \{ f\colon  \partial_x^\alpha f \in L^{2,s}(\mathbb R), \, 0\leq \alpha \leq k \}, \quad \quad s\in \mathbb R
$$
with norm $\| f \|_{H^{k,s}}^2 = \|  (1+|x|^2)^\frac{s}{2} f \|_{H^k}$
\item $\mathcal{B}(X,Y)$ denotes the space of bounded linear operators from $X$ to $Y$ and $\mathcal{B}(X)=\mathcal{B}(X,X)$. 
\item Summation notation:\ $\sum_\pm f_\pm \equiv f_+ + f_-$ .
\item The letter $C$ shall denote a generic constant. 
\end{enumerate}

\paragraph{Acknowledgements.} B. Osting was supported in part by US NSF Grant No. DMS06-02235, EMSW21- RTG: Numerical Mathematics for Scientific Computing. M.I. Weinstein was supported in part by  U.S. NSF Grants  DMS-07-07850 and  DMS-10-08855. The authors wish to thank P. Deift, T. Heinz,  A. Millis,  Y. Silberberg and C.W. Wong  for stimulating conversations.

\section{Spectral theory for the one-dimensional Schr\"odinger operator with compact potential}
\label{sec:Scattering}
In this section, we discuss relevant properties of the Schr\"odinger operator $H_V \equiv - \partial_x^2  + V$ for sufficiently regular and compactly supported potentials, {\it e.g.} $V\in \cA_1(a,b,\mu)$. More general and complete treatments can be found, for example, in \cite{Agmon75,tang-zworski-1dscat}. 

\subsection{The outgoing resolvent operator}
\label{sec:ResolventOp}
Let $0\ne k \in \mathbb C$. For $V=0$, we introduce the 
outgoing free resolvent
\begin{equation}
\label{eqn:resolvent}
\psi(x) = R_0[f](x,k)  = \int_{\mathbb R} G_0(x,y,k) f(y) \ud y, \quad  \quad
G_0(x,y,k) \equiv \imath (2 k)^{-1} \exp(\imath k |x-y|) 
\end{equation}
defined for $f\in L^2_{\text{comp}}(\mathbb R)$. The function $\psi=R_0(k)f$ satisfies the free Schr\"odinger equation and outgoing boundary condition
\begin{align*}
(-\partial_x^2- k^2) \psi = f, \ \ \ 
\lim_{x\rightarrow \pm \infty} (\partial_x \mp \imath k) \psi =0.
\end{align*}
For $V\neq0$ we introduce the outgoing resolvent, $R_V(k) \equiv (H_V - k^2)^{-1}$, satisfying  
\begin{equation}
\label{eqn:RVid}
(H_V - k^2)\circ R_V( k) = \text{Id}
\end{equation}
and which, for $\Im k>0$, is bounded on $L^2(\mathbb R)$ except for a discrete set of the form, $k_l=\imath{\kappa_l},\ \kappa_l>0$,
 where $-\kappa_l^2$ are eigenvalues of $H_V$.
We have the identity
$$
R_0 = R_0 \circ (H_V - k^2) \circ R_V = (\text{Id} + R_0 V) \circ R_V, 
$$ 
or equivalently, the Lipmann-Schwinger equation, 
\begin{equation}
\label{eqn:RV}
R_V = (\text{Id} + R_0 V)^{-1} \circ R_0,\ \ \Im k>0,\ \ k\ne \imath \kappa_l
\end{equation}
 
\begin{prop} 
\label{RVmero} 
The following are properties of the resolvent, $R_V$\ \cite{Agmon75,tang-zworski-1dscat}. 
\begin{enumerate}
\item The  family of operators $R_V(k)\colon L^2_{\text{comp}}(\mathbb R) \rightarrow L^2_{\text{loc}}(\mathbb R)$, given by Eq. (\ref{eqn:RV}), exists and has a meromorphic extension to $k \in \mathbb{C}$. It has no pole for $k \in \mathbb{R}\setminus \{0\}$. 
\item For $k\in\mathbb{R}$ and arbitrary $f\in L^2_{\text{comp}}$, the function $\psi=R_V(k)f$ satisfies 
\begin{equation*}
(H_V - k^2) \psi = f,\ \ \ \ \lim_{x\rightarrow \pm \infty} (\partial_x \mp \imath k) \psi =0.
\end{equation*}
\end{enumerate}
\end{prop}

\no We denote by $G_V(x,y,k)$, the Green's function, defined as the kernel of the integral operator $R_V(k)$, in analogy with Eq. (\ref{eqn:resolvent}). In the upper half plane, $\Im k>0$, $G_V(x,y,k)$ has a  finite number of simple poles at $k_l=i \kappa_l,\ \kappa_l>0$ . 
In the lower half plane,  $\Im k < 0$, $G_V(x,y,k)$ may have poles at {\it resonances}, values of $k$ for which the scattering resonance spectral problem:
 \begin{equation*}
(H_V - k^2) \psi = 0,\ \ \ \ \lim_{x\rightarrow \pm \infty} (\partial_x \mp \imath k) \psi =0.
\end{equation*}
has a non-trivial solution.

A consequence of Theorem 4.2 of \cite{Agmon75} is the following:
\begin{thm}[Limiting absorption principle]
\label{thm:lap}
For $\Im k>0$, the resolvent $R_V(k) = (H_V - k^2)^{-1}$ is a meromorphic function with values in $\mathcal{B}(L^2)$. For $s>\frac{1}{2}$, it can be extended to $\Im k\ge0$ as an operator on $\mathcal{B}(L^{2,s},H^{2,-s})$ with limit
\begin{equation}
R_V(k_0)\ =\ \lim_{\substack{\Im k>0, \\ k\to k_0\in\mathbb{R}}} R_V(k).
\end{equation}
Throughout this paper, we shall understand $R_V(k_0)$, for $k_0\in\mathbb{R}$, to be the limit taken in this way.
\end{thm}

Since we are interested in how the properties of solutions change with the potential, we make use of the resolvent identity 
\begin{equation}
\label{eqn:2ndId}
R_V - R_U =  R_V(U-V) R_U.
\end{equation}

We now refine Thm. \ref{thm:lap} by showing that $R_V(k)\colon L^{2,s} \rightarrow H^{2,-s}$ is (locally) Lipschitz continuous with respect to $V$. To prove this, we shall use  the following bounds, used in the proof of Theorem \ref{thm:lap} \cite{Agmon75}:
\begin{subequations}
\label{eqn:R0bnd}
\begin{align}
\| R_0(k) \|_{L^{2,s} \rightarrow H^{2,-s}} & \leq C \\
\| (\text{Id}+ R_0(k)V)^{-1} \|_{H^{2,-s} \rightarrow H^{2,-s}} & \leq C(V),\ \ \Im k\ge0,\ \ s>\frac{1}{2}
\end{align}
\end{subequations}
and the following 
\begin{lem}
\label{lem:compL2s}
Suppose $f\in L^{2,-s}(\mathbb R)$ has compact support with $\text{supp}(f) \subset [-a,a]$. Then $f\in L^{2,s}(\mathbb R)$ and
$$
\|f\|_{L^{2,s}} \leq C(a) a^{2s} \|f \|_{L^{2,-s}}.
$$
\end{lem}
\begin{proof}
$
\|f\|_{L^{2,s}}^2 \equiv \int f^2 (1+|x|^2 )^s \ud x \leq C(a) a^{4s} \int f^2 (1 + |x|^2)^{-s} \ud x = C(a) a^{4s} \|f\|_{L^{2,-s}}^2.
$
\end{proof}
\begin{prop} 
\label{prop:RnConv}
Fix $a,b,\mu \in \mathbb R$, $V \in \cA_1(a,b,\mu)$ and for $\rho>0$ denote by
\begin{equation}
\label{eqn:LinftyBall2}
B^\infty(V,\rho) = \{ U \in \cA_1(a,b,\mu) \colon \|V-U\|_\infty < \rho \}.
\end{equation}
There exists a $\rho_0>0$ such that  if $U \in B^\infty(V,\rho_0)$, then for $s>\frac{1}{2}$,
\begin{align}
\label{eqn:RnConv}
\| R_V(k) - R_U(k) \|_{L^{2,s}\rightarrow H^{2,-s}} \leq C(V,\rho_0,a) \| V-U \|_\infty
\end{align}
uniformly for all $k\in \mathbb R$.
\end{prop}

\begin{proof} 
Let $f\in L^{2,s}(\mathbb R)$. Using Eq. (\ref{eqn:2ndId}) and Thm. \ref{thm:lap}, we compute 
\begin{align*}
\| (R_V - R_U)f \|_{H^{2,-s}} \leq C(V) \|(U-V) R_U f  \|_{L^{2,s}}.
\end{align*}
Then using Lemma \ref{lem:compL2s} we have
\begin{align*}
\|(U-V) R_U f  \|_{L^{2,s}} 
&\leq C(a) a^{2s} \|U-V\|_{\infty} \|R_U f \|_{L^{2,-s}} \\
&\leq C(a) a^{2s} \|U-V\|_{\infty} \| R_U f \|_{H^{2,-s}}
\end{align*}
so that 
\begin{align*}
\| (R_V - R_U)f \|_{H^{2,-s}} \leq C(V,a)  \|U-V\|_{\infty} \|R_U\|_{L^{2,s}\rightarrow H^{2,-s}}  \|f\|_{L^{2,s}}.
\end{align*}
We now claim that there exists a $\rho_0>0$ and constants $C(V)$ and $C(V,a)$ such that for $U \in B^\infty(V,\rho_0)$ 
 \begin{align}
 \label{eqn:IpRU}
\| R_U \|_{L^{2,s} \rightarrow H^{2,-s}}  \leq C(V) \frac{1}{1-C(V,a) \rho_0}.
\end{align}
Equation (\ref{eqn:RnConv}) now follows once we have shown Eq. (\ref{eqn:IpRU}). 
To show Eq. (\ref{eqn:IpRU}), we use the resolvent identity
\begin{align*}
R_U = \left( \text{Id} + ( \text{Id} + R_0 V)^{-1} R_0 (U-V) \right)^{-1} R_V
\end{align*}
and Thm. \ref{thm:lap} to obtain
\begin{align}
\label{eqn:RUnorm}
\| R_U \|_{L^{2,s} \rightarrow H^{2,-s}} \leq C(V) \| \left( \text{Id} + ( \text{Id} + R_0 V)^{-1} R_0 (U-V) \right)^{-1}\|_{H^{2,-s} \rightarrow H^{2,-s}}
\end{align}
Using Eqs. (\ref{eqn:R0bnd}a) and (\ref{eqn:R0bnd}b) we have
$$
\|(\text{Id} + R_0 V)^{-1} R_0(U-V) \|_{H^{2,-s} \rightarrow H^{2,-s}} \leq C(V,a) \rho_0
$$
and Eq. (\ref{eqn:IpRU}) follows from using the  Neumann series in Eq. (\ref{eqn:RUnorm}).  
\end{proof}

\begin{prop} 
\label{prop:LipK}
Let $V \in \cA_1(a,b,\mu)$, $k\in \mathbb R$, $k\neq 0$, $s>\frac{1}{2}$. There exists a $\rho_0>0$ such that if $k'\in B(k,\rho_0)$ 
\begin{subequations}
\label{eqn:LipK}
\begin{align}
 \| R_0(k) - R_0(k') \|_{L^{2,s} \rightarrow H^{2,-s}} &\leq C(\rho_0, a) |k-k'| \\
 \| R_V(k) - R_V(k') \|_{L^{2,s} \rightarrow H^{2,-s}} &\leq C(\rho_0, V, a) |k-k'|
\end{align}
\end{subequations}
\end{prop}
\begin{proof} Eq.  (\ref{eqn:LipK}a) follows from Eq. \eqref{eqn:resolvent}. To show  Eq. (\ref{eqn:LipK}b), we compute
\begin{align}
\nonumber
 \| R_V(k) - R_V(k') \|_{L^{2,s} \rightarrow H^{2,-s}} 
  \leq &  \| \left[ (\text{Id} + R_0(k)V)^{-1} - (\text{Id} + R_0(k')V)^{-1} \right] R_0(k)   \|_{L^{2,s} \rightarrow H^{2,-s}} \\
\nonumber
 &+  \| (\text{Id} + R_0(k')V)^{-1} [ R_0(k) - R_0(k')] \|_{L^{2,s} \rightarrow H^{2,-s}} \\
\nonumber
  \leq & C  \| (\text{Id} + R_0(k)V)^{-1} - (\text{Id} + R_0(k')V)^{-1}   \|_{H^{2,-s} \rightarrow H^{2,-s}} \\
\label{eqn:diffk}
 &+  C(V) \| [ R_0(k) - R_0(k')] \|_{L^{2,s} \rightarrow H^{2,-s}} 
\end{align}
where we used Eq. \eqref{eqn:R0bnd}. 
We now use the resolvent identity
\begin{align*}
(\text{Id} + R_0(k) V)^{-1} - (\text{Id} + R_0(k') V)^{-1} = (\text{Id} + R_0(k) V)^{-1} \left[ R_0(k) - R_0(k') \right] V (\text{Id} + R_0(k') V)^{-1}
\end{align*}
and Eq. \eqref{eqn:R0bnd} on the first term in Eq. \eqref{eqn:diffk} to obtain
\begin{align*}
\| (\text{Id} + R_0(k)V)^{-1} - (\text{Id} + R_0(k')V)^{-1}   \|_{H^{2,-s} \rightarrow H^{2,-s}}  \leq C(V) \| \left( R_0(k) - R_0(k') \right)  \|_{H^{2,-s} \rightarrow H^{2,-s}}. 
 \end{align*}
Now applying Eq. (\ref{eqn:LipK}a) to Eq. \eqref{eqn:diffk} yields Eq.  (\ref{eqn:LipK}b) as desired.
\end{proof}

\subsection{Distorted plane waves, $e_{V\pm}(x;k)$, and Jost solutions, $f_{V\pm}(x;k)$}

Distorted plane waves, $e_{V\pm}(x;k)$, are states which explicitly encode the scattering experiment of a plane wave incident on a potential
 resulting in reflected and transmitted waves. The Jost solutions, $f_{V\pm}(x;k)$, can be thought of as the states to which $e^{\pm ikx}$ deform for nonzero $V(x)$  in the spectral decomposition of $H_V$. In this section, we introduce these states and give their basic properties.
 
The continuous spectrum of $H_{V}$ is $\sigma_c(H_{V}) = [0,\infty)$. Corresponding to each point $k^2 \in \sigma_c(H_{V})$ are two \emph{distorted plane waves} $e_{V\pm}(x,k)$ satisfying 
\begin{subequations}
\label{eqn:GenEigHV}
\begin{align}
H_V e_{V\pm}(x,k) =  k^2 e_{V\pm}(x,k) \\
\lim_{x\rightarrow \pm \infty} (\partial_x \mp \imath k) e_V(x,k) = 0. 
\end{align}
\end{subequations}
For $V=0$ these are the plane wave solutions $e_{0\pm}(x,k) = e^{\pm \imath k x}$. 
For $V\neq 0$, the unique solution to Eq. (\ref{eqn:GenEigHV}) is given by 
\begin{equation}
\label{eqn:contStates}
e_{V\pm}(x,k) = e^{\pm \imath k x} - R_V[V e_{0\pm}(\cdot,k) ](x,k). 
\end{equation}
If $V$ is compactly supported within $[-a,a]$,  for $x\notin [-a,a]$, the solutions $e_{V\pm}(x,k)$ are given in terms of the transmission $t_{V}(k)$ and reflection $r_{V}(k)$ coefficients 
\begin{align}
\label{eqn:defTransRefl}
e_{V+}(x,k) &=  
\begin{cases}
e^{\imath k x} + r_{V}(k) e^{-\imath k x}, & x<- a \\
t_{V}(k) e^{\imath k x}, & x>a
\end{cases} 
\end{align}
For $k\ne0$, we have  $|r_V(k)|^2 + |t_V(k)|^2=1$. 
 If $V$ is a symmetric, then $e_{V-}(x,k) =  e_{V+}(-x,k)$.
\medskip

The following proposition establishes that if $V$ is compactly supported then $|t_V(k)|$ is bounded away from zero, uniformly in $k$. We shall use 
this result in the interpretation of our numerical computations in Section \ref{sec:2mechs}.
\begin{prop}
\label{prop:tkGreaterZero}
Suppose $\text{supp}(V)\subset [-a,a]$, $k\neq 0$
\begin{equation}
 |t_V(k)| \geq \exp\left( - \min\{1/|k|, \, 2a\}\ \int_{-a}^a |V(s)| \ud s \right).
\label{tbound}\end{equation}
\end{prop}
\begin{proof}
Consider the integral equation governing $e_{V+}(x,k)$:
$$
e_{V+}(x,k) = t_V(k) e^{\imath k x} - \int_x^a \frac{\sin k (x-y)}{k} V(y) e_{V+}(y,k) \ud y, \ x<a.
$$
For $x\ge a$, $e_{V+}(x,k) = t_V(k) e^{\imath k x} $.
Since  $k^{-1}\sin(k(x-y))$ is bounded by $\min\{|k|^{-1},|x-y|\}$ we have  
\begin{equation}
| e_{V+}(x,k) |  \leq  |t_V(k)| +  \int_x^a \min\{|k|^{-1},|x-y|\}\ |V(y)|\  |e_{V+}(y,k)| \ud y
\label{bd1}\end{equation}
Therefore, by Gronwall's inequality
\begin{align}
| e_{V+}(x,k) | & \leq  |t_V(k)|\ \exp\left(\int_x^a \min\{|k|^{-1},|x-y|\}\  |V(y)|  \ud y\right)\nn\\
 & \le\  
|t_V(k)|\ \exp\left(\min\{|k|^{-1},2a\}\ \int_{-a}^a \  |V(y)|  \ud y\right),\ \ x<a,
\label{bd2}\end{align}
and thus
\begin{equation}
|t(k)| \geq | e_{V+}(x,k) |  \exp\left( -\  \min\{|k|^{-1},2a\}\ \int_{-a}^a\  |V(y)|  \ud y \right),\ x<a.
\label{bd3}\end{equation}
To bound $| e_{V+}(x,k) |$, observe that  for fixed $k\neq 0$, we can choose $x^*< -a$ such that $\arg(r_V(k)) = 2k x^*$. Therefore
\begin{equation}
 | e_{V+}(x^*,k) | = | e^{\imath k x} + r_V(k) e^{-\imath k x} | =  | 1 + r_V(k) e^{-2\imath k x^*}| = \left| 1 + |r_V(k)| \right| \geq 1 .
\label{bd4}\end{equation}
The bounds \eqref{bd3} and \eqref{bd4} imply \eqref{tbound}.
\end{proof}

The following proposition states that we can choose a constant to bound the distorted plane waves for all potentials in a small  $L^{\infty}$-neighborhood of a $V\in \cA_1$.
\begin{prop}
\label{prop:eVbounded}
Fix $a,b,\mu \in \mathbb R$ and $V\in\cA_1(a,b,\mu)$ and let $B^\infty(V,\rho)$ be as in Eq. (\ref{eqn:LinftyBall2}). 
There exists a $\rho_0>0$ such that for $U\in B^\infty(V,\rho_0)$ the distorted plane waves $e_{U\pm}(x,k)$ satisfy 
$$
\|e_{U\pm}(\cdot,k)\|_{L^{\infty}([-a,a])} \leq C(a, V, \rho_0).
$$ 
\end{prop}
\begin{proof}
Using Eq. (\ref{eqn:contStates}), we compute 
\begin{align*}
\| R_U[U e^{\imath k x}] \|_{L^\infty([-a,a])} 
&\leq C(a) \| (1+|x|^2)^{-s} R_U[U e^{\imath k x}] \|_{L^\infty} \\
&\leq C(a) \| (1+|x|^2)^{-s} R_U[U e^{\imath k x}] \|_{H^2} \\
&= C(a) \| R_U[U e^{\imath k x}] \|_{H^{2,-s}} \\
&\leq C(a,V,\rho_0)
\end{align*}
This last line follows from a Proposition \ref{prop:RnConv}. 
\end{proof}
\medskip

\begin{defn}\label{Jost}
The {\it Jost solutions}, $f_{V\pm}(x,k)$, associated with the time-independent Schr\"odinger equation $(H_V-k^2)u=0$ are defined by
\begin{equation}
e_{V+}(x;k)\ =\ t_V(k)\ f_{V+}(x;k),\ \ \ \ e_{V-}(x;k)\ =\ t_V(k)\ f_{V-}(x;k),
\label{Jost}\end{equation}
where $ f_{V+}(x;k)\sim e^{ikx}$ as $x\to+\infty$ and $ f_{V-}(x;k)\sim e^{-ikx}$ as $x\to-\infty$.
\end{defn}

By results of \cite{DT:79}, for any $k\in\mathbb{R}$ and any compact subset, $C$, of $\mathbb{R} $
\begin{equation}
\max_{x\in C}|f_{V\pm}(x;k)|\le K_{k,C}<\infty
\end{equation}
Note also that Propositions \ref{prop:eVbounded} and \ref{prop:tkGreaterZero} imply a bound on $|f_{V\pm}|$ in the case where $V$ has compact support.

\subsection{Spectral decomposition of the 1D Schr\"odinger operator}

We state the spectral theorem in terms of the distorted plane waves (see {\it e.g.} \cite{tang-zworski-1dscat}):
\begin{prop}[Spectral Decomposition]
\label{prop:SpecDecomp}
Let $e_{V\pm}$ and $f_{V\pm}$ denote the distorted plane waves and Jost solutions given by \eqref{eqn:contStates}
 and \eqref{Jost}. Let $\lambda_j$ for $j=1\ldots N$ be the eigenvalues of $H_V$ with corresponding (normalized) eigenfunctions $\psi_j(x)$. Then, $h = P_d h\ +\ P_c h$
where $P_d$ and $P_c$ are, respectively, projections onto the discrete and continuous spectral parts of $H_V$ given by
\begin{align}
\label{eqn:specHV}
P_c h &= \frac{1}{2\pi}\int_0^{\infty} \left[\ \left(e_{V+}(\cdot, k) , h\right) e_{V+}(x, k)
 + \left(e_{V-}(\cdot, k) , h \right) e_{V-}(x, k)\  \right]\ \ud{k}\nn\\
 &= \frac{1}{2\pi}\int_0^{\infty} \left[\ \left(f_{V+}(\cdot, k) , h \right) f_{V+}(x, k)
 + \left(f_{V-}(\cdot, k) , h \right) f_{V-}(x, k)\  \right]\ |t_V(k)|^2\ \ud{k}\\
 P_d h &=  \sum_{j = 1}^{N} \lambda_j  (\psi_j, \, h) \psi_j(x)\nn
 \end{align}
 Moreover, 
 \begin{align}
 g(H_V) h &=  \frac{1}{2\pi}\int_0^{\infty} g(k^2)\ \left[\ \left(f_{V+}(\cdot, k) , h\right) f_{V+}(x, k)
 + \left(f_{V-}(\cdot, k) , h \right) f_{V-}(x, k)\  \right]\ |t_V(k)|^2\ \ud{k}\nonumber\\
 &+ \sum_{j = 1}^{N} g(\lambda_j)  (\psi_j, \, h) \psi_j(x),
\label{spec-borel}\end{align}
where $g$ is any Borel function. Finally,  by approximation we have that \eqref{spec-borel} holds with $g(\zeta)=\delta(\zeta)$, the Dirac delta distribution
 in the distributional sense. 
\end{prop}
\no

\section{Radiation damping and Fermi's Golden Rule}
\label{sec:FGR-ionize}
In this section, we explain how $\Gamma[V]$, given in Eq.  \eqref{toy-fgr}, emerges as the key quantity controlling the lifetime of the metastable state. 
We now state a theorem on the ionization and decay of the bound state and then sketch the idea of a proof, which explains the mechanism of decay and  \eqref{toy-decay}. A detailed proof can be found in  \cite{Soffer-Weinstein:98,Kirr-Weinstein:01,Kirr-Weinstein:03}. 
The following result holds for generic potentials with one bound state. In particular, these hypotheses are satisfied by $V\in \cA_1^\reg(a,b,\mu)$.
\begin{thm}\label{thm:SW}
Consider the parametrically forced Schr\"odinger equation
\begin{equation}
     \imath \D_t\phi^\epsilon\ =\ H_V\phi^\epsilon\ +\ \epsilon\ \cos(\mu t)\ \beta(x)\ \phi^\epsilon.
     \label{f-schrod}
     \end{equation}
Assume $V$ and $\beta$ satisfies the general conditions of \cite{Soffer-Weinstein:98,Kirr-Weinstein:01}.  Consider the initial value problem for Eq. \eqref{forced-schrod} with $\phi^\epsilon(x,0) = \phi_0\in L^{2,\sigma}(\mathbb R),$ where $\sigma\ge 1$. 
Assume 
\begin{enumerate}
\item $k_V^2 \equiv \lambda_V+\mu>0$ (resonance with the continuum at $\mathcal{O}(\epsilon^2)$)
\item $\Gamma[V]>0$, where $\Gamma[V]$ is the non-negative quantity defined by  
\begin{subequations}
\label{eqn:Gamma}
\begin{align}
\Gamma[V] &\equiv \frac{\pi}{4}\ \langle \beta \psi_V, \delta(H_V - k_{V}^2) P_c \beta \psi \rangle \\
& = \frac{1}{16\ k_V} \sum_\pm |\langle \beta \psi_V, e_{V\pm}(\cdot,k_{V}) \rangle  |^2\\
& = \frac{1}{16\ k_V}\ |t_V(k_V)|^2\ \sum_\pm |\langle \beta \psi_V, f_{V\pm}(\cdot,k_{V}) \rangle  |^2,
\end{align}
where $e_{V\pm}$ and $f_{V\pm}$ denote, respectively, the distorted plane wave and Jost solutions, and 
$t_V(k)$ denotes the transmission coefficient.
\end{subequations}
\end{enumerate}
Then, there exists $\epsilon_0>0$ such that for $\epsilon<\epsilon_0$
\begin{align*}
\left| \langle \psi_V, \phi^\epsilon(\cdot,t)  \rangle \right| &\sim \left| \langle \psi_V, \phi_0  \rangle \right|e^{-\epsilon^2 \Gamma[V] t}\ + \ \mathcal{O}(\epsilon), \quad && 0\le t\le\cO(\epsilon^{-2})\\
\left\|  \phi^\epsilon(\cdot,t)  \right\|_{L^{2,-\sigma}}  & \lesssim \ t^{-\frac{1}{2}}  \| \phi^\epsilon_0  \|_{L^{2,\sigma}},   && t\gg1.
\end{align*}
\end{thm}

\begin{remark}
For  certain choices of potentials, the parametrically forced Schr\"odinger (ionization) problem is exactly solvable by Laplace transform methods and the  time-behavior can be computed for  {\it all}  $\epsilon>0$. See, for example, \cite{Costin-Lebowitz-Rokhlenko:00,Costin:2001fk}. 
\end{remark}

 \medskip \nit{\bf A sketch of the proof.} In this sketch, we drop the subscript on $\psi_V$ and  superscript on $\phi^\epsilon$. For small $\epsilon$, it is natural to decompose the solution as
\begin{equation}
\phi(t,x) = a(t) \psi(x) + \phi_c(t,x)
\label{decomp}\end{equation}
where $a(t) = \langle \psi, \phi(\cdot,t) \rangle$ and $\phi_c = P_c [\phi]$ is the continuum projection; see
 \eqref{eqn:specHV}. To simplify the discussion we take as initial data:
\begin{equation}
a(0)=a_0,\ \ \ \ \phi_c(0,x)\equiv0.
\label{data}
\end{equation}
Substitution of \eqref{decomp} into \eqref{f-schrod} and projecting onto the discrete and continuous spectral parts of $H_V$ yields the following  coupled system:
\begin{subequations}
\label{eqn:weakCoupledSys}
\begin{align}
(\imath \partial_t - \lambda) a(t) &= \epsilon \cos(\mu t) \langle \psi, \beta \psi \rangle a(t)  + \epsilon \cos(\mu t) \langle \psi, \beta \phi_c \rangle \label{first} \\
(\imath \partial_t - H_V) \phi_c &= \epsilon \cos(\mu t) P_c [\beta \psi] a(t) + \epsilon \cos(\mu t) P_c[\beta \phi_c] .
\label{second}\end{align}
\end{subequations}
Since  $\epsilon$ has been assumed small, the coupling between $a(t)$ and $\phi_c(t,x)$ is weak. We now proceed to make a set of simplifications leading to a minimal model,
 in which the mechanism of radiation damping is fairly transparent. First, since the first term on the right hand side
 of \eqref{first} contributes an order $\epsilon$ mean-zero frequency shift from $\lambda$, we neglect it. Second, from equation \eqref{second} we formally have that $\phi_c=\mathcal{O}(\epsilon)$. Therefore, the last term on the right hand side of \eqref{second} is $\mathcal{O}(\epsilon^2)$ and we therefore neglect it. Finally, the second equation evolves in the continuous spectral part of $H_V$ and we formally replace $H_V$ by $H_0=-\Delta$.

The resulting system is the  following Hamiltonian system of an oscillator of complex amplitude $a(t)$ coupled to a field $\phi_c(t,x)$:
\begin{subequations}
\label{eqn:coupledOscillatorWave}
\begin{align}
(\imath \partial_t - \lambda) a(t) &= \epsilon \cos(\mu t) \langle \psi, \beta \phi_c \rangle \\
(\imath \partial_t + \Delta) \phi_c &= \epsilon \cos(\mu t) \beta \psi a(t) .
\end{align}
\end{subequations}
We can exploit a separation of time-scales by extracting the fast phase from $a(t)$ via the  substitution
$$
a(t) = e^{- \imath \lambda t} A(t),
$$
giving the following equation for the slowly varying amplitude, $A(t)$:
\begin{equation}
\label{eqn:A(t)}
\imath \partial_t  A(t) = \epsilon \cos(\mu t) \langle \psi, \beta \phi_c \rangle  e^{\imath \lambda t}
\end{equation}
Now, Duhamel's formula is used to rewrite  Eq. (\ref{eqn:coupledOscillatorWave}b) as 
$$
\phi_c(t) =  - \imath \epsilon \int_0^t e^{\imath \Delta (t-s)} \cos(\mu s) \beta \psi a(s) \ud s
$$
since $\phi_c(0) = 0$. We insert this back into Eq. \eqref{eqn:A(t)} to obtain the closed equation for $A(t)$. 
$$
 \partial_t  A(t) = - \epsilon^2 \cos(\mu t) e^{-\imath \lambda t} A(t) \int_0^t \langle \beta \psi ,   e^{\imath \Delta (t-s)} \beta \psi \rangle   \cos(\mu s)  e^{-\imath \lambda s}A(s) \ud s  
$$
Writing $\cos(\mu t) = \frac{1}{2} \left( e^{\imath \mu t} + e^{-\imath \mu t}\right)$, we find that if $k_{res}^2 \equiv \lambda + \mu > 0$, then it is a resonant frequency and 
\begin{align*}
 \partial_t  A(t) &\approx - \frac{1}{4}\epsilon^2 e^{-\imath k_{res}^2 t} A(t) \int_0^t \langle \beta \psi ,   e^{\imath \Delta (t-s)} \beta \psi \rangle  e^{-\imath k_{res}^2 s}A(s) \ud s    \\
 &\approx\ - \frac{1}{4}\epsilon^2 \langle \beta \psi , (-\Delta - k_{res}^2-\imath 0)^{-1} \beta \psi \rangle A(t)  
 \end{align*}
Here, $(-\Delta - E-\imath 0)^{-1}=\lim_{\delta\downarrow0}(-\Delta - E^2-\imath \delta)^{-1}$. The choice of regularization is dictated by the outgoing radiation condition for $t\to+\infty$; see \cite{Soffer-Weinstein:98,Kirr-Weinstein:01}. 
\medskip

Returning to the original (un-approximated) equations  \eqref{eqn:weakCoupledSys}, we have analogously 
\begin{equation}
 \partial_t  A(t) \approx - \frac{1}{4}\epsilon^2 \langle \beta \psi , (H_V - k_{res}^2-\imath 0)^{-1} P_c [\beta \psi] \rangle A(t)\ \equiv -\epsilon^2(\Lambda +i\Gamma) A(t).
 \label{toy-derived}
\end{equation}
The coefficient of $A(t)$ in \eqref{toy-derived} can be computed by applying the functional calculus identity \eqref{spec-borel} 
 to the function $g(s)= (s-k_{res}^2-\imath \tau)^{-1}$, together with the distributional identity
 $$\lim_{\tau\downarrow0}(s-k_{res}^2-\imath \tau)^{-1}= P.V. \ (s-k_{res}^2)^{-1}+\imath \pi\ \delta(s-k_{res}^2)$$
 and the identification $s\to H_V$. In particular, 
 \begin{align}
 \Gamma[V] &= \frac{1}{4}\cdot \frac{1}{2\pi} \langle\ \beta\psi_V,\delta(H_V-k_{res}^2)\ P_c\beta\psi\ \rangle\nn\\
  & = \frac{1}{8\pi}\ \int_0^\infty\ \delta(k^2-k_{res}^2)\ \left[\ \left|\langle f_{V+}(\cdot,k),\beta\psi_V\rangle\right|^2
  +\ \left|\langle f_{V-}(\cdot,k),\beta\psi_V\rangle\right|^2\ \right]\ \left|t_V(k)\right|^2\ dk,
  \nn\end{align}
  from which the expression \eqref{eqn:Gamma} follows after setting $\nu=k^2$ and carrying out the integral.

\medskip

\section{A constrained optimization problem: design of a potential to  minimize radiative loss}
\label{sec:PDP}
We now consider the Potential Design Problem (PDP) given in Eq. (\ref{toy-opt}) with $\Gamma[V]$ defined in Eq. (\ref{eqn:Gamma}). 
We begin by discussing the set of admissible potentials $\mathcal{A}_1(a,b,\mu)$ defined in Def. \ref{def:VAdmiss}.
For the purpose of numerical computation we relax he admissible set, $\mathcal{A}_1\to\cA_1^{\reg}$, by replacing the discrete constraint ($H_V$ has exactly one eigenvalue) by an inequality constraint, in terms of a regularization parameter, $\reg$.
We then show that the objective function is locally Lipschitz and that a solution to the PDP exists in the modified admissible set, $\cA_1^{\reg}(a,b,\mu)$. 

\subsection{The admissible set $\mathcal{A}_1$ and its relaxation, $\mathcal{A}_1^\reg$ }
\label{sec:AdmPot}
Denote the Wronskian of the distorted plane waves, $e_{V\pm}$, by
\begin{align}
\label{eqn:Wron}
W_V(k) \equiv \text{Wron}(e_{V+}(\cdot,k), e_{V-}(\cdot,k)), \quad \quad k \in \mathbb C. 
\end{align}
Zeros of $W_V(k)$ correspond to poles of the Green's function $G_V(x,y,k)$ as introduced in Sec. \ref{sec:ResolventOp}. In particular, the zeros of $W_V(k)$, in the upper half plane, are eigenvalues. 
The number of eigenvalues is increased or decreased  by one, typically through the crossing of a  simple 
zero of $W_V(k)$ through $k=0$ as $V$ varies.\footnote{ Potentials, $V$, for which $W_V(k=0) = 0$ are called \emph{exceptional}. The value 
 $k=0$, corresponding to edge of the continuous spectrum is then called a zero energy resonance
or a half-eigenvalue with half-bound state $e_{V\pm}(x,0)$\  \cite{RS3}.} 
 Our strategy to fix the number of eigenvalues is then
 to start with a one bound state potential and deform $V$, keeping $W_V(0)\ne0$.
However, numerically it is advantageous to replace contraint $W_V(0)\ne0$ by the inequality constraint $W_V(0)^2\ge\reg$.
For $\reg>0$, we regularize  $\cA_1(a,b,\mu)$ by introducing
\begin{equation}
\label{eqn:Aeps}
\cA^\reg_1(a,b,\mu) \equiv \cA_1(a,b,\mu) \cap \{V\colon W_V(0)^2 \geq \reg \}. 
\end{equation}

\begin{rem}
\label{lem:nonConvex}
Note that the set of admissible potentials $\cA_1^\reg(a,b,\mu)$ is not convex. Indeed, 
 counter-examples can be explicitly generated and are illustrated by the following cartoon superposition of
potential wells at sufficiently separated points \cite{Harrell}:
\end{rem}

$$
0.5 \, \,
\begin{sparkline}{10}
\spark 0 1 1 1 /
\spark .5 0 .5 1.2 /
\spark 0.00 0.99 0.03 0.98 0.07 0.97 0.10 0.95 0.14 0.92 0.17 0.86 0.21 0.77 0.24 0.64 0.28 0.46 0.31 0.26 0.34 0.10 0.38 0.05 0.41 0.14 0.45 0.31 0.48 0.51 0.52 0.68 0.55 0.80 0.59 0.88 0.62 0.93 0.66 0.96 0.69 0.98 0.72 0.99 0.76 0.99 0.79 1.00 0.83 1.00 0.86 1.00 0.90 1.00 0.93 1.00 0.97 1.00 1.00 1.00 /
\end{sparkline}
\quad + \quad 0.5 \,\,
\begin{sparkline}{10}
\spark 0 1 1 1 /
\spark .5 0 .5 1.2 /
\spark 0.00 1.00 0.03 1.00 0.07 1.00 0.10 1.00 0.14 1.00 0.17 1.00 0.21 1.00 0.24 0.99 0.28 0.99 0.31 0.98 0.34 0.96 0.38 0.93 0.41 0.88 0.45 0.80 0.48 0.68 0.52 0.51 0.55 0.31 0.59 0.14 0.62 0.05 0.66 0.10 0.69 0.26 0.72 0.46 0.76 0.64 0.79 0.77 0.83 0.86 0.86 0.92 0.90 0.95 0.93 0.97 0.97 0.98 1.00 0.99 /
\end{sparkline}
\quad = \quad  
\begin{sparkline}{10}
\spark 0 1 1 1 /
\spark .5 0 .5 1.2 /
\spark 0.00 1.00 0.03 0.99 0.07 0.99 0.10 0.98 0.14 0.96 0.17 0.93 0.21 0.88 0.24 0.82 0.28 0.72 0.31 0.62 0.34 0.53 0.38 0.49 0.41 0.51 0.45 0.56 0.48 0.60 0.52 0.60 0.55 0.56 0.59 0.51 0.62 0.49 0.66 0.53 0.69 0.62 0.72 0.72 0.76 0.82 0.79 0.88 0.83 0.93 0.86 0.96 0.90 0.98 0.93 0.99 0.97 0.99 1.00 1.00 /
\end{sparkline}
$$
The two potentials on the left hand side of the equation support a single bound state while the convex combination on the right supports two. 
\bigskip

\begin{lem}
\label{lem:Aeps}
For $\reg>0$, let $\gamma \colon [0,1] \rightarrow \cA_1^{\reg}(a,b,\mu)$ be a smooth function valued path. If $H_{\gamma(0)}$ supports a single bound state then so does $H_{\gamma(t)}$ for $t\in [0,1]$.
\end{lem} 
\begin{proof}
We show that no bound states are lost along the path $\gamma(t)$. A similar argument shows that no bound states are gained.

Define $f(k,t) \equiv W_{\gamma(t)}(k)$ and consider the equation $f(k,t) = 0$. Let $\lambda_0$ denote the eigenvalue of $H_{\gamma(0)}$, {\it i.e.} 
$f(\imath \sqrt{|\lambda_0|}, 0) = 0$. Since $\partial_k f(k,t) \Big|_{(\imath \sqrt{|\lambda_0|},0)} \neq 0$, {\it i.e.} $\lambda_0$ is a simple eigenvalue, by the implicit function theorem, there exists $T>0$ such that for $|t|<T$, there is a parameterized family $t\mapsto \lambda_t$ with $\lambda_t<0$ and 
$f(\imath\sqrt{|\lambda_t|},t) = 0$.  
Let 
$$
t^\# = \sup \{0\le t\le 1\colon f(\imath\sqrt{|\lambda_t|},t) = 0 \text{ and } \lambda_t<0 \}\ >\ 0.
$$
If $\lambda_{t^\#} = 0,\ t^\#<1$ then $f(0,t^\#)=W_{\gamma(t^\#)}(0) = 0$. This contradicts $\gamma(t) \subset  \cA_1^{\reg}(a,b,\mu)$ Therefore, $t^\#=1$. 
\end{proof}

Let $\eta_V^\pm(x) = e_{V\pm}(x, 0)$ be the \emph{distorted plane waves} at $k=0$ which satisfy
\begin{align}
\label{eqn:halfBoundState}
H_V \eta_V^\pm = 0, \quad  \quad \lim_{x\rightarrow \pm \infty} \eta_V^\pm = 1;
\end{align}
see equation \eqref{eqn:defTransRefl}.

Our gradient-based optimization approach requires that we compute the variation of the Wronskian, $W_V(0)$, with respect to the potential $V$. This calculation will also be used to establish Lipschitz continuity of $W_V(0)$. 
\begin{prop} 
\label{prop:wronDeriv}
Let $\eta_V^\pm(x)$ satisfy Eq. (\ref{eqn:halfBoundState}). 
The Fr\'echet derivative  of the Wronskian $W_V(0) = \text{Wron}(\eta_V^+, \eta_V^-) \colon L^2_{\text{comp}} \rightarrow \mathbb R$ with respect to the potential is given by 
$$
\frac{\delta W_V(0)}{\delta V}= - \eta_V^+ \eta_V^-  .
$$
\end{prop}
\begin{proof} See Appendix \ref{sec:appCompGrads}. \end{proof}

\begin{rem}
\label{rem:dWdVsym}
If $V$ is symmetric, then $\frac{\delta W_V(0)}{\delta V}$ is symmetric. 
\end{rem}
To prove that $W_V(0)$ is locally Lipschitz, we use the following lemma.
\begin{lem}
\label{lem:DiffImplyCont}
Let $f[V]\colon L^2([-a,a]) \rightarrow \mathbb R$ be a Fr\'echet differentiable functional with 
$$
f[U] = f[V] + \Big \langle \frac{\delta f}{\delta V} \Big|_{V} , U-V \Big \rangle  + o( \| U - V\|_2 ).
$$
Suppose further that the variation $\frac{\delta f}{\delta V}$ is bounded in an $L^\infty$-neighborhood of $V$.  
Then there exists a $\rho_0 > 0$ and a constant $C(\rho_0,V,a)$ such that for $U\in B^{\infty}(V,\rho_0)$
$$
| f[U] - f[V] | \leq C(\rho_0,V,a) \| U - V\|_{L^\infty([-a,a])}
$$
\end{lem}
\begin{proof}
The Mean Value Theorem and Proposition \ref{prop:wronDeriv} imply that there exists a $\rho_0>0$ such that for every $U\in B^\infty(V,\rho_0)$, there exists a potential 
$ \tilde V = t V + (1-t) U$ 
for some $t\in [0,1]$ such that 
$$
f[U] = f[V] + \Big\langle \frac{\delta f}{\delta V}\Big|_{\tilde V},  U-V \Big\rangle\ .
$$
This gives the estimate
\begin{align*}
| f[U] - f[V] | \leq  \|  U - V \|_\infty  \left| \int_{-a}^a  \frac{\delta f}{\delta V}\Big|_{\tilde V} \ud x \right|.
\end{align*}
The proof is completed by  choosing $C(\rho_0,V,a) = \sup_{\tilde V \in B^{\infty}(V,\rho_0)} \left| \int_{-a}^a  \frac{\delta f}{\delta V}\big|_{\tilde V} \ud x \right|$.
\end{proof}
\begin{prop}[local Lipschitz continuity of $W_V(0)$]
\label{prop:WLip}
 Fix $a,b,\mu \in \mathbb R$, $\reg> 0$, and let $V\in\cA_1^\reg(a,b,\mu)$. For $\rho>0$ denote by
\begin{equation}
\label{eqn:LinftyBall}
B^\infty(V,\rho) = \{ U \in \cA_1^\reg(a,b,\mu) \colon \|V-U\|_\infty < \rho \}
\end{equation}
the $L^\infty(\mathbb R)$ ball around $V$ in $A_1^\reg(a,b,\mu)$. Let $W_V(0)$ be as defined in Eq. (\ref{eqn:Wron}).
There exists $\rho_0>0$ and a constant $C(\rho_0,V,a)$ such that if $U \in B^\infty(V,\rho_0)$
then 
$$
| W_U(0) - W_V(0)|  \leq C(\rho_0,V,a) \| U - V\|_\infty.
$$
\end{prop}
\begin{proof} 
Propositions \ref{prop:eVbounded} and \ref{prop:wronDeriv} give that $\frac{\delta W_V}{\delta V} = -\eta^+_{V} \eta^-_{V}$ is pointwise bounded in a neighborhood of $V$. 
The result now follows immediately from Lemma \ref{lem:DiffImplyCont}.  
\end{proof}


\subsection{Properties of the objective functional, $\Gamma[V]$}
In this section, we begin with a formal calculation of the Fr\'echet derivative of $\Gamma[V]$, given by \eqref{eqn:Gamma}.
We then show that $\Gamma[V]$ is (locally) Lipschitz with respect to $V$. 

\begin{prop}
\label{prop:grad}
Let $L>a$. The Fr\'echet derivative of $\Gamma[V]\colon L^2_{\text{comp}}([-a,a]) \rightarrow \mathbb R$ given in Eq. (\ref{eqn:Gamma}) 
with respect to the potential V is given by 
\begin{subequations}
\begin{align}
\label{eqn:chainRuleGamma}
\frac{\delta \Gamma}{\delta V} &=
 \frac{\delta \Gamma}{\delta \psi_V} \left[ \frac{\delta \psi_V}{\delta V} [ \delta V ] \right]
+ \sum_\pm  \frac{\delta \Gamma}{\delta e_{V\pm}} \left[   \frac{\delta e_{V\pm}(\cdot, k_V)}{\delta V} [\delta V] \right]
 - \frac{\Gamma}{k_V} \langle \frac{\delta k_V}{\delta V}, \delta V \rangle  \\
\label{eqn:grad}
&= - \frac{1}{8k_V} \psi_V R_V(\sqrt{\lambda_V}) P_c \left[ \Re \sum_\pm \langle e_{V\pm} , \beta \psi_V \rangle \beta e_{V\pm} \right] 
 - \frac{\Gamma}{2k_V^2} \psi_V^2 \\
&+ \sum_\pm \frac{1}{8k_V} \Re  \langle e_{V\pm}, \beta \psi_V  \rangle   
\left(  \frac{1}{2k_V}\langle \beta \psi_V,  A_\pm \rangle \psi_V^2  - \overline{e_{V\pm}} R_V(k_V)\left[ \beta \psi_V\right] \right)
\nonumber
\end{align}
\end{subequations}
where
\begin{align}
\label{eqn:Apm}
A_\pm(x) =& \pm \imath x e^{\imath k_V x} - R_V(k_V)\left[ 2k_V \phi_{V\pm} \pm \imath x V e^{\pm \imath k_V x} \right] \\
&+ \frac{e^{\imath k_V L}}{2 k_V}\left( \phi_{V\pm}(-L) e_{V+}(x,k_V) + \phi_{V\pm}(L) e_{V-}(x,k_V) \right),
\nonumber
\end{align}
$e_{V\pm} = e_{V\pm}(\cdot, k_V)$,  
and $\phi_{V\pm}(x) \equiv e^{\pm \imath k_V x} - e_{V\pm}(x,k_V)$ satisfies Eq. (\ref{eqn:phi}). 
\end{prop}
\begin{proof} Equation (\ref{eqn:chainRuleGamma}) is obtained by the chain rule. A detailed computation of each term is given in  Appendix \ref{sec:appCompGrads}. \end{proof}

\begin{rem}
\label{rem:dGdVsym}
If the potential and $\beta$ are symmetric, so is $\frac{\delta \Gamma}{\delta V}$. 
\end{rem}

\begin{prop}[local Lipschitz continuity of $\Gamma$ ]
\label{prop:GLip} 
 Fix $a,b,\mu\in \mathbb R$, $\reg> 0$ and $V\in\cA_1^\reg(a,b,\mu)$ and define $B^\infty(V,\rho)$ as in Eq. (\ref{eqn:LinftyBall}). 
There exists $\rho_0>0$ and a constant $C(V,\rho_0,a)$ such that if $U \in B^\infty(V,\rho) $ then 
$$
|\Gamma[U] - \Gamma[V]|  \leq C(V,\rho_0,a) \|U-V\|_{\infty}. 
$$ 
\end{prop}
\begin{proof}
First we use the triangle inequality to obtain
\begin{align}
\nonumber
\big| \Gamma[U] - \Gamma[V]  \big| 
& = \Big|  \sum_\pm \frac{1}{16k_U} |\langle \beta \psi_U, e_{U\pm}(\cdot, k_U)\rangle|^2 
 - \frac{1}{16k_V} |\langle \beta \psi_V, e_{V\pm}(\cdot, k_U)\rangle|^2 \Big| \\
\nonumber
& \leq \frac{1}{16}  \sum_\pm 
\Big| \frac{1}{k_U}-\frac{1}{k_V}\Big|\ \  |\langle \beta \psi_U, e_{U\pm}(\cdot, k_U)\rangle|^2 
 \\
\nonumber
&\quad \quad+ \frac{1}{16k_V} \ \sum_\pm\ \left[\ \Big| |\langle \beta \psi_U, e_{U\pm}(\cdot, k_U)\rangle|^2  
 -  |\langle \beta \psi_V, e_{U\pm}(\cdot, k_U)\rangle|^2 \Big|\ \right. \\
\nonumber
&\quad \quad + \Big|  |\langle \beta \psi_V, e_{U\pm}(\cdot, k_U)\rangle|^2
-  |\langle \beta \psi_V, e_{V\pm}(\cdot, k_U)\rangle|^2 \Big| \\
\nonumber
&\quad\quad + \left.  \Big|  |\langle \beta \psi_V, e_{V\pm}(\cdot, k_U)\rangle|^2 
 -  |\langle \beta \psi_V, e_{V\pm}(\cdot, k_V)\rangle|^2 \Big|\ \right] \\
 &\equiv \sum_\pm A_\pm +B_\pm +C_\pm + D_\pm
 \label{eqn:GammaDiff}
\end{align}
We now treat the terms $A_\pm-D_\pm$ in Eq. (\ref{eqn:GammaDiff}) in turn. We'll repeatedly 
 use the inequality $\big| |a|^2 - |b|^2 \big| \leq |a+b||a-b|$.

\paragraph{A.}
We compute
\begin{align*}
A_\pm =  \frac{|k_V - k_U|}{16k_Vk_U}|\langle \beta \psi_U, e_{U\pm}(\cdot, k_U)\rangle|^2.
\end{align*}
Recalling $k_V = \sqrt{\lambda_V + \mu}$ and using Eq. (\ref{eqn:dlamdV}) we have
\begin{align}
|k_V - k_U| &\leq \frac{1}{2k_V} |\lambda_V - \lambda_U| + o(|\lambda_V - \lambda_U|) \nonumber \\
&\leq  \frac{1}{2k_V} \langle \psi_V^2, |U-V| \rangle + o(\|U-V\|_\infty) \nonumber \\
&\leq  \frac{1}{2k_V}\|U-V\|_\infty + o(\|U-V\|_\infty). \label{kUVdiff}
\end{align}

\paragraph{B.} We compute
\begin{align*}
B_\pm &=  \frac{1}{16k_V} \Big| |\langle \beta \psi_U, e_{U\pm}(\cdot, k_U)\rangle|^2  
 - |\langle \beta \psi_V, e_{U\pm}(\cdot, k_U)\rangle|^2 \Big| \\
 &\leq  \frac{1}{16k_V} |\langle \beta (\psi_U+ \psi_V), e_{U\pm}(\cdot, k_U)\rangle| |\langle \beta (\psi_U- \psi_V), e_{U\pm}(\cdot, k_U)\rangle| \\
 &\leq  \frac{1}{16k_V} |\langle \beta (\psi_U+ \psi_V), e_{U\pm}(\cdot, k_U)\rangle|  \| e_{U\pm}(\cdot, k_U) \|_{L^2(K)} \| \beta (\psi_U- \psi_V) \|_2.
\end{align*}
Now using Eq. (\ref{eqn:dpsidV}) and Thm. \ref{thm:lap} we obtain 
\begin{align*}
\| \psi_U- \psi_V \|_2   &\leq \| R_V(\sqrt{\lambda}) P_c [\psi_V(U-V)] \|_2 + o(\|U-V\|)_{\infty} \\
&\leq C(V) \|U-V\|_\infty + o(\|U-V\|)_{\infty} . 
\end{align*}

\paragraph{C.} We compute
\begin{align*} 
C_\pm &=  \frac{1}{16k_V} \Big| |\langle \beta \psi_V, e_{U\pm}(\cdot, k_U)\rangle|^2  
 - |\langle \beta \psi_V, e_{V\pm}(\cdot, k_U)\rangle|^2 \Big| \\
 &\leq \frac{1}{16k_V} |\langle \beta \psi_V, e_{U\pm}(\cdot, k_U) + e_{V\pm}(\cdot, k_U) \rangle|
 |\langle \beta \psi_V, e_{U\pm}(\cdot, k_U) - e_{V\pm}(\cdot, k_U) \rangle| 
\end{align*}
Using  Eq. (\ref{eqn:contStates}) we have that 
 \begin{align*}
e_{U\pm}(\cdot, k_U) - e_{V\pm}(\cdot, k_U) &= -R_U(k_U)[Ue^{\pm\imath k_U x}] + R_V(k_U)[Ve^{\pm \imath k_U x}] \\
&= R_U(k_U)[(V-U) e^{\pm\imath k_U x}] + (R_V(k_U) - R_U(k_U)) [Ve^{\pm \imath k_U x}]
 \end{align*}
Fact: For $K\subset B(0,r)$ compact we have
\begin{align}
\nonumber
\|f\|_{L^\infty(K)}  &= \| (1+|x|^2)^{-\frac{s}{2}} f (1+|x|^2)^{\frac{s}{2}} \|_{L^\infty(K)}  \\
\nonumber
&\leq C_r \| (1+|x|^2)^{-\frac{s}{2}} f  \|_{L^\infty(\mathbb R)}  \\
\nonumber
&\leq C_r \| (1+|x|^2)^{-\frac{s}{2}} f  \|_{H^1(\mathbb R)}  \\
\label{eqn:H1weighted}
& = C_r \| f  \|_{H^{1,-s} (\mathbb R)}.
\end{align}
Thus, by Prop. (\ref{prop:RnConv}) and Eq. (\ref{eqn:H1weighted}) we have 
\begin{align*}
&\| e_{U\pm}(\cdot, k_U) - e_{V\pm}(\cdot, k_U) \|_{L^\infty(K)} \\
&\leq C_r \| e_{U\pm}(\cdot, k_U) - e_{V\pm}(\cdot, k_U) \|_{H^{1,-s} (\mathbb R))} \\
&\leq C_r\left(  \| R_V\|_{L^{2,s}\rightarrow H^{2,-s}} \| (V-U) e^{\pm \imath k_U x} \|_{L^{2,s}}  + \| R_V - R_U \|_{L^{2,s}\rightarrow H^{2,-s}} \| V e^{\pm \imath k_U x} \|_{L^{2,s}} \right)  \\
&\leq C_r\left(  \| R_V\|_{L^{2,s}\rightarrow H^{2,-s}} \|e^{\pm \imath k_U x} \|_{L^{2,s}(K)}  + C(V,\rho_0) \|V e^{\pm \imath k_U x} \|_{L^{2,s}(K)} \right) \|V-U\|_{\infty}
\end{align*}

\paragraph{D.} We compute
\begin{align*}
D_\pm &=  \frac{1}{16k_V} \Big| |\langle \beta \psi_V, e_{V\pm}(\cdot, k_U)\rangle|^2  
 - |\langle \beta \psi_V, e_{V\pm}(\cdot, k_V)\rangle|^2 \Big| \\
 &\leq  \frac{1}{16k_V} |\langle \beta \psi_V, e_{V\pm}(\cdot, k_U) +e_{V\pm}(\cdot, k_V)\rangle|
  |\langle \beta \psi_V, e_{V\pm}(\cdot, k_U) - e_{V\pm}(\cdot, k_V)\rangle|
\end{align*}
But, 
\begin{align} 
 e_{V\pm}(\cdot, k_U) - e_{V\pm}(\cdot, k_V) = & e^{\pm \imath k_V x}\left( e^{\pm \imath (k_U - k_V)x} -1 \right) - R_V(k_U) \left[ V e^{\pm\imath k_Vx} \left( e^{\pm \imath (k_U - k_V)x} -1 \right)\right] \nonumber\\
 & + \left( R_V(k_V) - R_V(k_U) \right) [Ve^{\pm \imath k_V x}] \label{but}
\end{align}
Insertion of \eqref{but} into the above bound on $D_\pm$ and use of Prop. \ref{prop:LipK}, \eqref{kUVdiff}
and Proposition \ref{prop:eVbounded} implies $|D_\pm|\le C\|V-U\|_\infty$.\\

Proposition \ref{prop:GLip} now follows from assembling the estimates {\bf A}-{\bf D}.
\end{proof}

\begin{rem}\label{rem:beta2V}
As mentioned in Remark \ref{rem:2classes}, 
we shall consider optimization problems where (i) $\beta(x)$ is a fixed specified function 
and where (ii) $\beta(x)=V(x)$.  For type (ii) problems  
$\Gamma[V]$ is given by
\begin{equation}
\Gamma[V] = \frac{1}{16 k_V} \sum_\pm |\langle V \psi_V, e_{V\pm}(\cdot,k_{V}) \rangle  |^2 .
\label{betaVgam}
\end{equation}
That $\Gamma[V]$ in this case is Lipschitz follows by the same arguments as above. Furthermore, it is  
Fr\'echet differentiable and the additional contribution of the Fr\'echet derivative of $\Gamma[V]$, to be added to the expression in
 Proposition \ref{prop:grad}, is given by:
$$
\frac{1}{8k_V} \psi_V(x)\  \Re \sum_\pm e_{V\pm}(x,k_V)  \langle  e_{V\pm}(\cdot ,k_V) , V \psi_V \rangle. 
$$
\end{rem}

\subsection{Existence of a minimizer}
We show that the potential design problem attains a minimum in the admissible class $\cA^\reg_1(a,b,\mu)$. Define
\begin{align}
\gamma_*^\reg(a,b,\mu) = \inf \{ \Gamma[V] \colon V \in \mathcal A_1^\reg(a,b,\mu) \} \geq 0. 
\end{align}

\begin{prop}
\label{thm:existence}
 There exists $V_* \in \mathcal A_1^\reg(a,b,\mu)$ such that $\Gamma[V_*] = \gamma_*^\reg(a,b,\mu) $. 
\end{prop}
\begin{proof} 
Let $\{V_n\} \subset \mathcal A_1^\reg(a,b,\mu)$ be a minimizing sequence, {\it i.e.}  $\lim_{n\uparrow \infty} \Gamma[V_n] = \gamma_*$. Since $\|V_n\|_{H^1(\mathbb{R})}\leq b$, there is a weakly convergent subsequence converging to $V_* \in H^1$. Moreover,  the family $\{V_n\} $ is uniformly bounded and equicontinuous on $[-a,a]$. By the Arzel\`a-Ascoli theorem, there is  a subsequence (which we continue to denote by $\{ V_n \}$) converging to $V_* \in H^1$ and 
 such that $V_n \rightarrow V_*$ uniformly on $[-a,a]$. 
By Prop. \ref{prop:WLip}, $W_V(0)$ is continuous with respect to $V$ on $\mathcal A_1^\reg(a,b,\mu)$, implying $V_*\in \mathcal A_1^\reg(a,b,\mu)$. 
By Prop. \ref{prop:GLip}, $\Gamma[V]$ is continuous on $\mathcal A_1^\reg(a,b,\mu)$ and 
\begin{align*}
\Gamma[V_*] = \lim_{n\uparrow \infty} \Gamma[V_n] = \gamma_* \ .
\end{align*}
\end{proof}

\begin{rem} 
Since $\mathcal{A}_1^\reg$ is not convext   (Remark \ref{lem:nonConvex}), uniqueness of the minimizer is not guaranteed. 
\end{rem}

\begin{cor}
\label{cor:betaEqualV}
By Remark \ref{rem:beta2V}, a minimizer also exists if we take $\Gamma$ with $\beta = V$. 
\end{cor}

\section{Numerical solution of the optimization problem}
\label{sec:numMeth}
In this section, we discuss a numerical solution of the potential design problem (PDP)
 \begin{equation}
 \min_{V\in\cA_1^\reg(a,b,\mu)} \Gamma\left[V\right]
 \label{eqn:PDP}
 \end{equation}
for fixed $a,b,\mu,\reg,$ and $ \beta(x)$,  where $\Gamma[V]$ is given in Eq. \eqref{eqn:Gamma} and $\cA_1^\reg(a,b,\mu)$ in Eq. \eqref{eqn:Aeps}. 

\paragraph{Forward Problem.} We refer to the evaluation of $\Gamma[V]$ and $W_V(0)$ for a given $V\in\cA_1^\reg(a,b,\mu)$ as the forward problem. 

To evaluate the objective function $\Gamma[V]$, first the eigenpair $(\lambda_V,\psi_V)$ satisfying Eq. (\ref{psi0lam0}) is computed using a three-point finite difference discretization and the Matlab \verb+eigs+ command. Next, the distorted plane waves $e_{V\pm}(x,\sqrt{\mu + \lambda_V})$ are computed using the decomposition as in 
Eq. (\ref{eqn:eVinTermsOfPhi}) and then solving Eq. (\ref{eqn:phi}) using the same discretization. The integrals for $\Gamma[V]$ are then evaluated using the trapezoidal rule.

The evaluation of $W_V(0)$ requires the distorted plane waves at $k=0$, which are computed using a Crank-Nicholson method. For numerical stability, the Wronskian, which is analytically a constant in $x$, is computed on a uniform grid and is averaged over the spatial domain. As a check on the discretization, we ensure that the variance of the Wronskian does not exceed a specified tolerance. 

\paragraph{Optimization Problem.} Local optima of Eq. (\ref{eqn:PDP}) are found using a line-search based L-BFGS quasi-Newton interior-point method \cite{NW2006} as implemented in the Matlab command \verb+fmincon+. 
We use the {\it optimize-then-discretize approach}, where gradients are computed as in Proposition \ref{prop:wronDeriv} and \ref{prop:grad} and  evaluated using the discretized counterparts.
The constraints, 
\begin{subequations}
\label{eqn:constraints}
\begin{align}
\lambda+\mu &\geq 0 \\
W_V(0)^2&\geq \reg \\
\|V\|_{H^1} &\leq b  
\end{align}
\end{subequations} 
are enforced using a logarithmic-barrier function. The method terminates when the line search cannot find a sufficient decrease in the objective function. 

In the numerical experiments, presented in Section \ref{sec:Results}, we use a computational domain larger than the interval $[-a,a]$ defining the support of the potential $V$ and depending on the magnitude of $a$, between 1000 and 3000 grid points. The method converges in less than 100 iterations and takes approximately 5-20 minutes using a 2.4 GHz dual processor machine with 2GB memory.
 
\paragraph{Time-dependent simulations.}
In Sections \ref{timeDependentSim} and \ref{sec:Filtering} we study time evolution for the initial value problem in Eq. \eqref{forced-schrod}.
This is accomplished using the same discretization as above and the time stepping routine for stiff ordinary differential equations implemented in \verb+ode15s+ in Matlab. The outgoing boundary conditions are approximated using a large domain with a dissipative term localized at the boundary. 


\section{Results of numerical experiments}
\label{sec:Results}

In this section, we present the results of many numerical experiments using the methods described in Section \ref{sec:numMeth} to study locally optimal solutions of the potential design problem (\ref{eqn:PDP}). 
The constraints in Eq. (\ref{eqn:PDP}) depend on $\mu$ (forcing frequency), $a$ (support width), $b$ ($H^1$ bound on $V$), and $\reg$ (relaxation parameter), while the objective function depends on the choice of  spatial perturbation of the potential, $\beta(x)$. For $\reg$ sufficiently small and $b$ sufficiently large, we find that in all numerically computed solutions of Eq. (\ref{eqn:PDP}), a local optimum is achieved at an interior point of the constraint set, $\mathcal{A}_1^{\reg}(a,b,\mu)$, {\it i.e.} the constraints  (\ref{eqn:constraints}) are not active at the optimal solution. 
This is in contrast to the structure of optimal solutions of other design problems studied in 
\cite{Cox1990,II:1986uq,OptimizationOfScatteringResonances,Kao:2005pt,Krein:1955ye,osher-santosa-jcp01} where the optimal potentials always attain the bounds and are referred to as ``bang bang'' controls.

We conjecture that the $H^1$ bound on $V$ can be relaxed in Proposition \ref{thm:existence} and that the constraints  of a compactly supported potential with a finite number of  bound states is sufficient for the minimization to be well posed, {\it i.e.}
there exists $b_0, \reg_0>0$ such that for $b\ge b_0$ and $\reg<\reg_0$: 
$$\min_{V\in\mathcal{A}_1^{\reg}(a,b,\mu)}\ \Gamma[V]\ =\ \min_{V\in\mathcal{A}_1^{0}(a,\infty,\mu)}\  
\Gamma[V].$$ 
Thus, we consider potential optimization problems for the two classes of $\beta(x)$ in Remark \ref{rem:2classes}
and vary $\mu$ and $a$. 



\subsection{Optimal potentials for varying support size, $a$,\\ with forcing frequency $\mu=2$ and $\beta(x) = \mathbf{1}_{[-2,2]}(x)$}
\label{sec:increaseSupport}

\begin{figure}[t!]
\centering
\begin{tabular}{cc}
\includegraphics[height=2.2in]{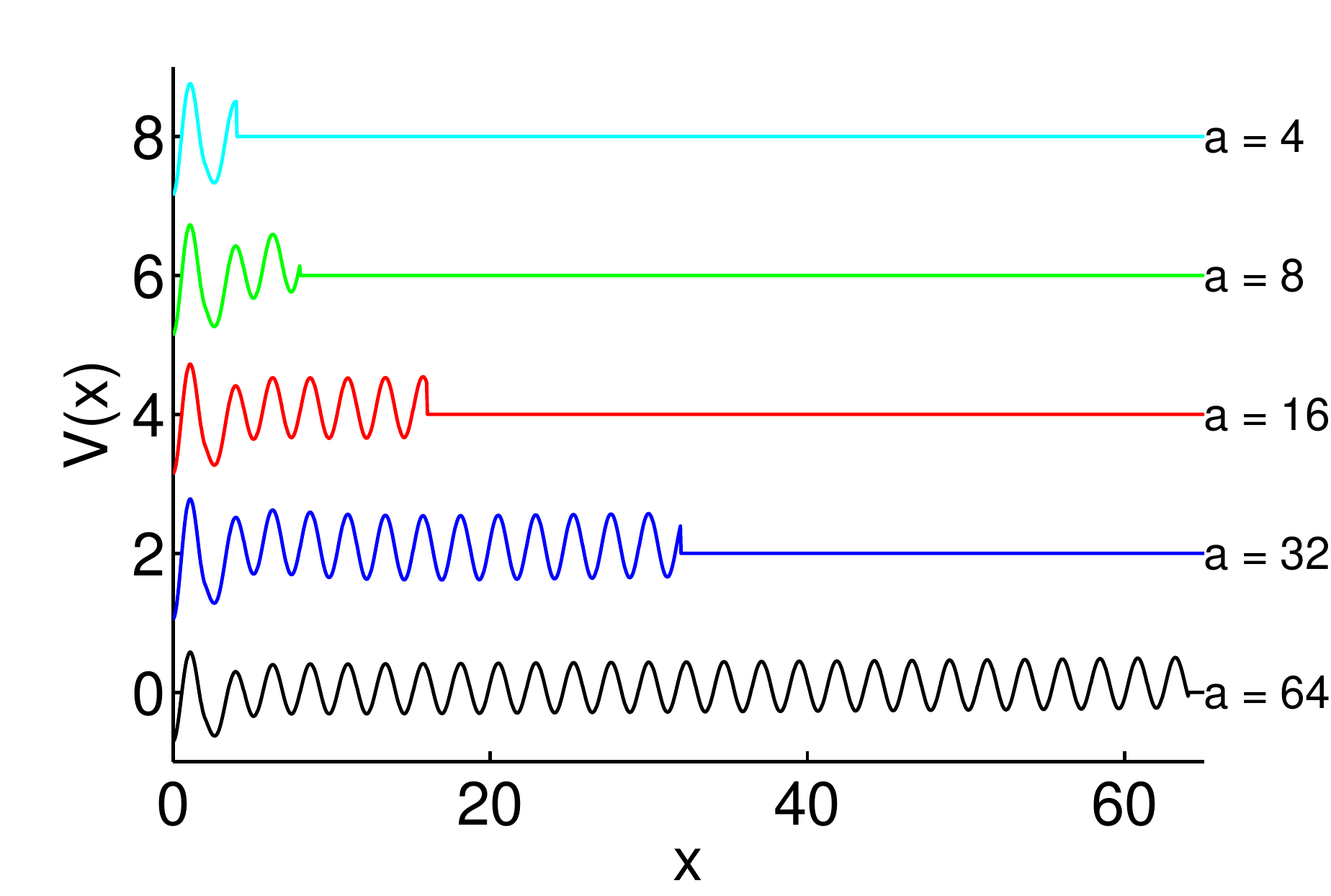}&
\includegraphics[height=2.2in]{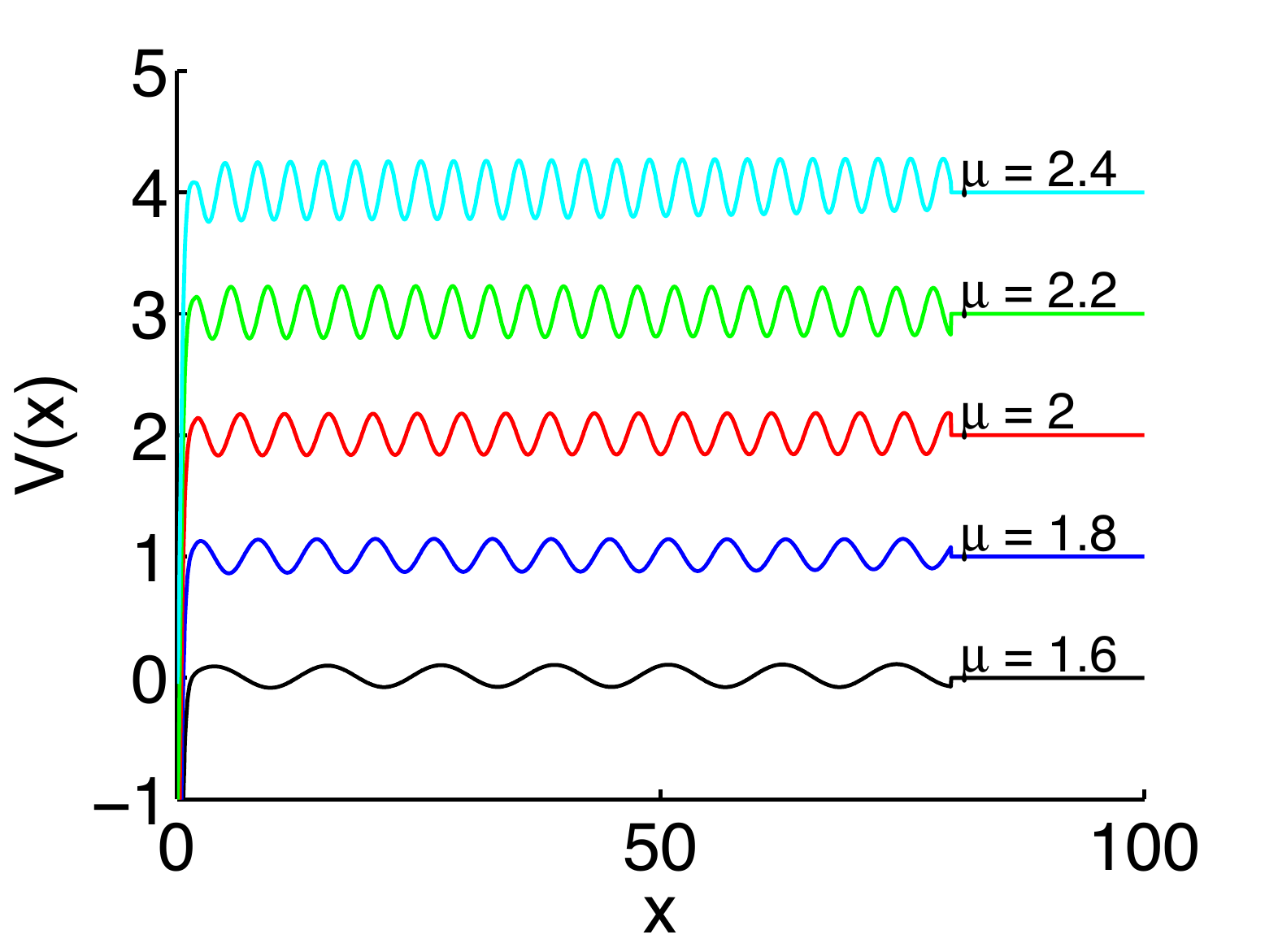} 
\end{tabular}
\caption{Locally optimal potentials for varying values of support $a$, with fixed frequency $\mu=2$ (left) and varying forcing frequency $\mu$, with fixed support $a=80$ (right). The potentials are symmetric in $x$, only $x\geq 0$ is plotted. }
\label{fig:compareParams}
\end{figure}

In  Fig. \ref{fig:compareParams} (left) we plot locally optimal potentials for 5 different values of the support, $a$. 
 (The potentials for different values of $a$ are shifted vertically.) 
The potentials that emerge are symmetric in $x$  (see remarks \ref{rem:dWdVsym} and \ref{rem:dGdVsym}) and periodic on the interval $[-a, a]$ with a defect at the origin. 
Let $V^*_a$ denote the optimal potential for support parameter $a$. We note that existing structure changes very little as we increase $a$. That is, $V^*_a$ and $V^*_b$ are nearly equivalent on the set $[c,c]$ where $c = \min(a,b)$. 
This is numerical support for Conjecture \ref{conj:convPer}. 

The following table gives the value of $\Gamma[V]$ for the sequence of potentials in Fig. \ref{fig:compareParams}(left).
\begin{center}
\begin{tabular}{c| l l l l l}
$a$ & 4 & 8 & 16 & 32 & 64\\
\hline
$\Gamma$ &  $3\times 10^{-10}$ & $2\times 10^{-11}$ & $1 \times 10^{-9}$ & $3\times 10^{-9}$ & $8\times10^{-10}$ \\
\end{tabular}
\end{center}
The non-monotonicity of $\Gamma$ with increasing support parameter, $a$, 
reflects the fact that we are only able to compute local minima of $\Gamma$. 
For small $a$, a relatively small number of  design variables give good numerical accuracy in evaluating the oscillatory integral that defines $\Gamma$ and the optimization method converges easily. 
However, as $a$ increases, the numerical method degrades as we are forced to balance accuracy 
with the number of optimization variables.

\subsection{Optimal potentials for varying  forcing frequency, $\mu$,\\
with fixed support size, $a=80$, and $\beta=\mathbf{1}_{[-2,2]}$}\label{sec:varymu}

In Fig. \ref{fig:compareParams}(right) we plot locally optimal potentials for 5 different values of forcing frequency $\mu$. The optimal potentials vary smoothly as we change $\mu$ with the period of the oscillation in the tails of the potentials decreasing with increasing $\mu$. 

The following table gives the value of $\Gamma[V]$ for the potentials in Fig. \ref{fig:compareParams}(right).
\begin{center}
\begin{tabular}{c| l l l l l}
$\mu$ & 1.6 & 1.8 & 2 & 2.2 & 2.4 \\
\hline
$\Gamma$ & $2\times 10^{-9}$  & $7\times 10^{-9}$ & $2\times 10^{-8}$ & $4\times 10^{-8}$  & $2\times 10^{-8}$ 
\end{tabular}
\end{center}

\subsection{Two mechanisms for potentials attaining small $\Gamma$}
\label{sec:2mechs}

\begin{figure}[t!]
\centering
\begin{tabular}{cc}
Low DOS Mechanism- $V_{A,opt}$ & Cancellation Mechanism- $V_{B,opt}$\\
\hline
\includegraphics[width=2in]{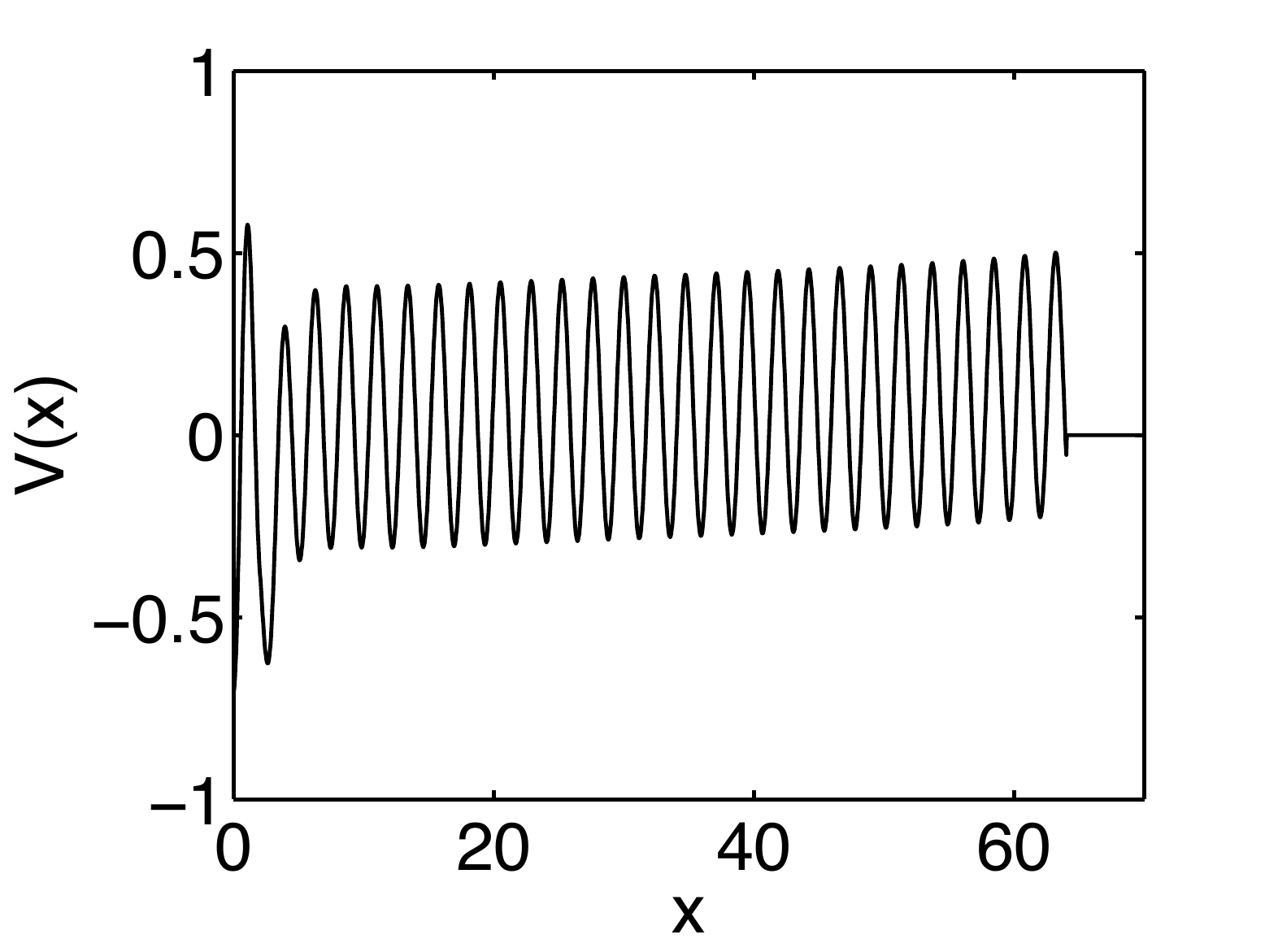} &
\includegraphics[width=2in]{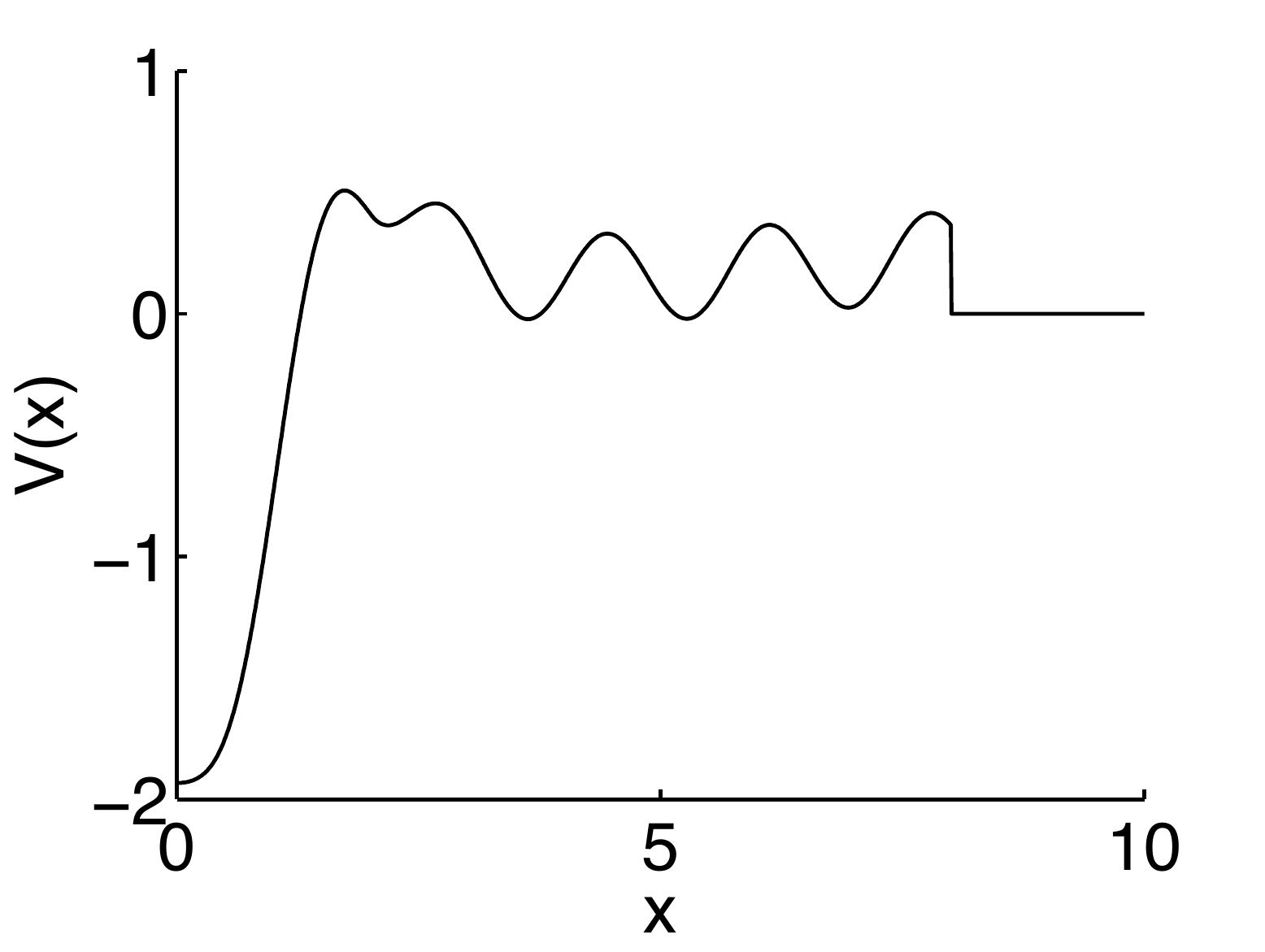}\\ 
\includegraphics[width=2in]{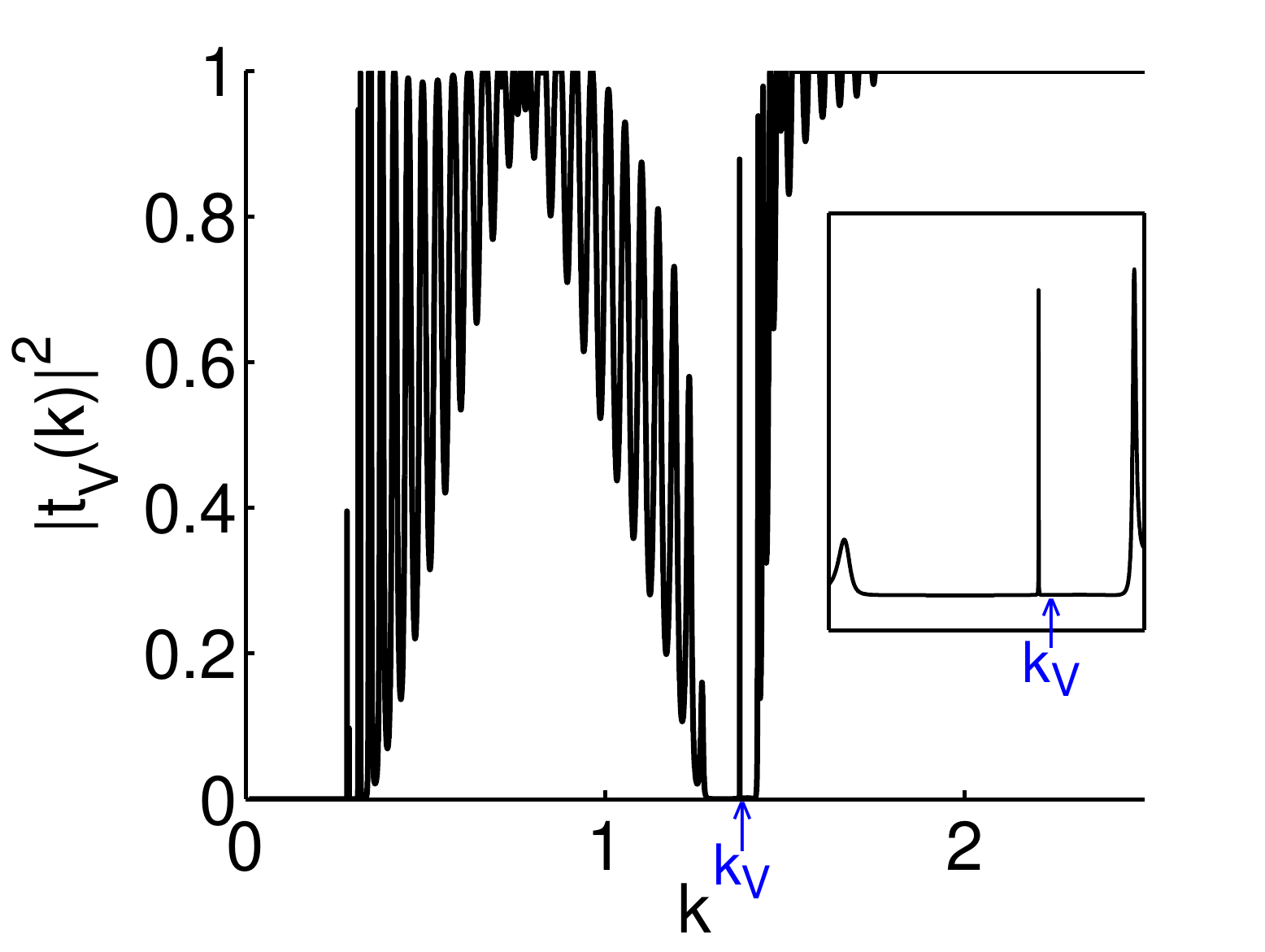} &
\includegraphics[width=2in]{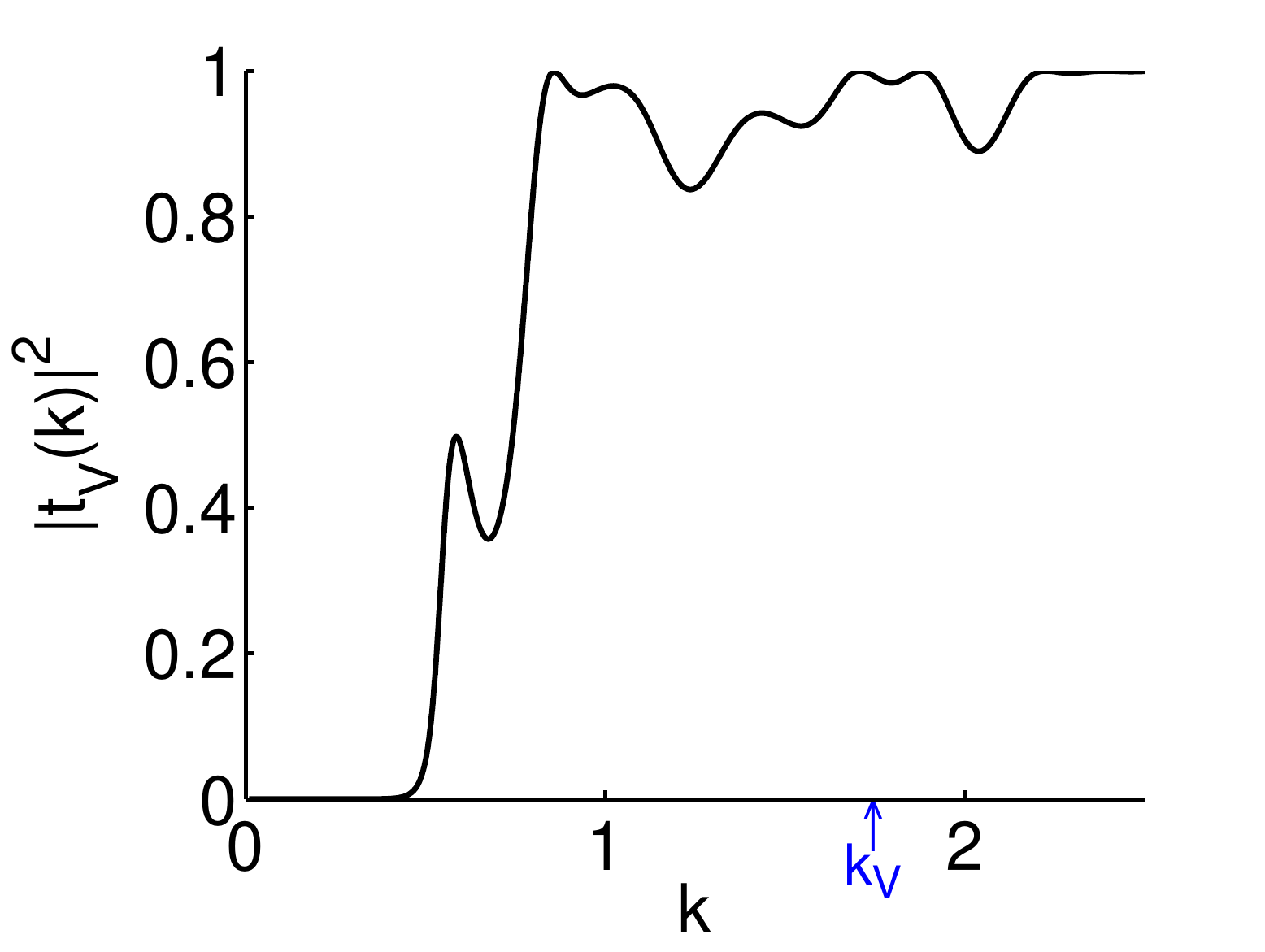}\\ 
\includegraphics[width=2in]{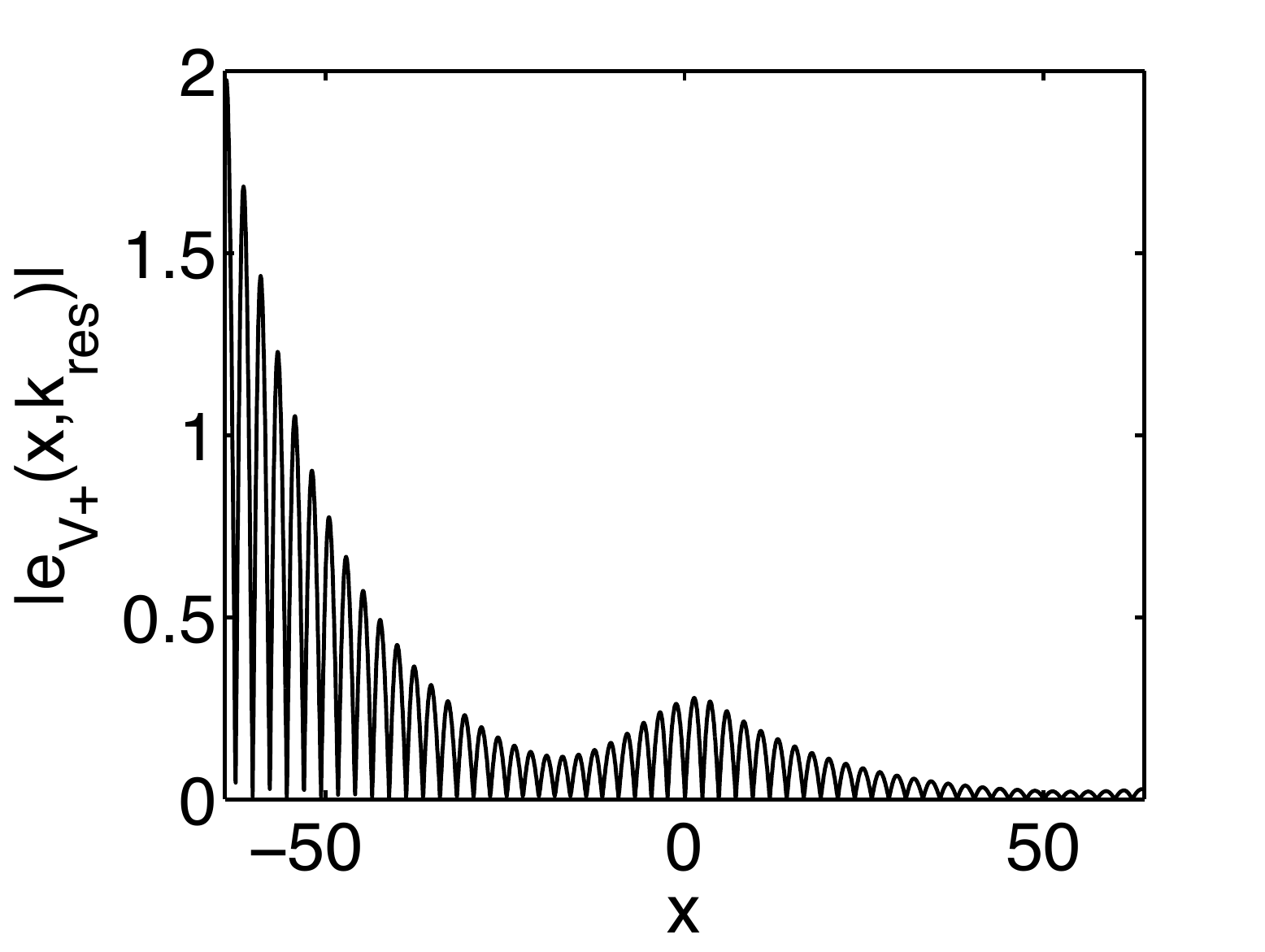} &
\includegraphics[width=2in]{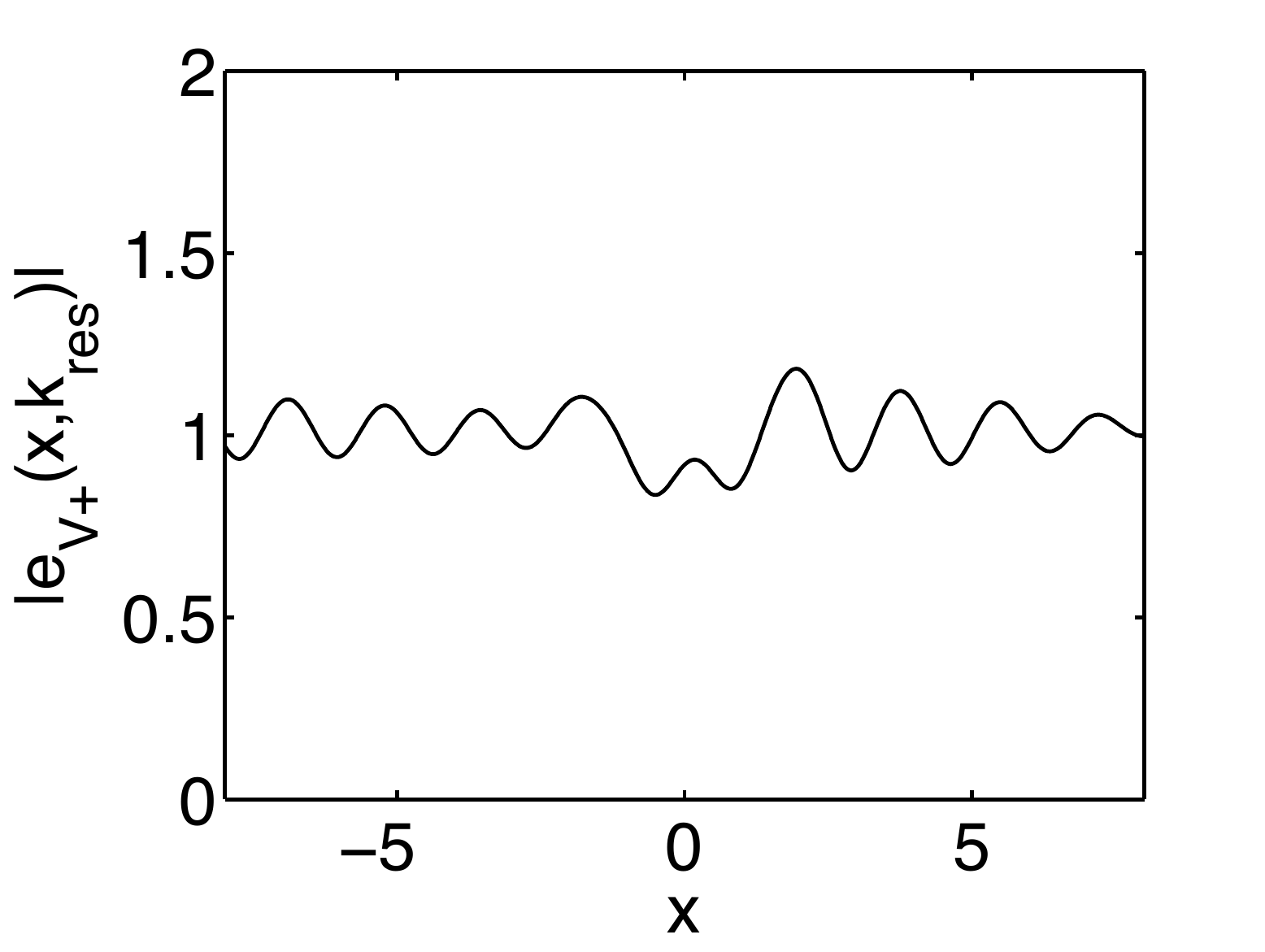}\\
\includegraphics[width=2in]{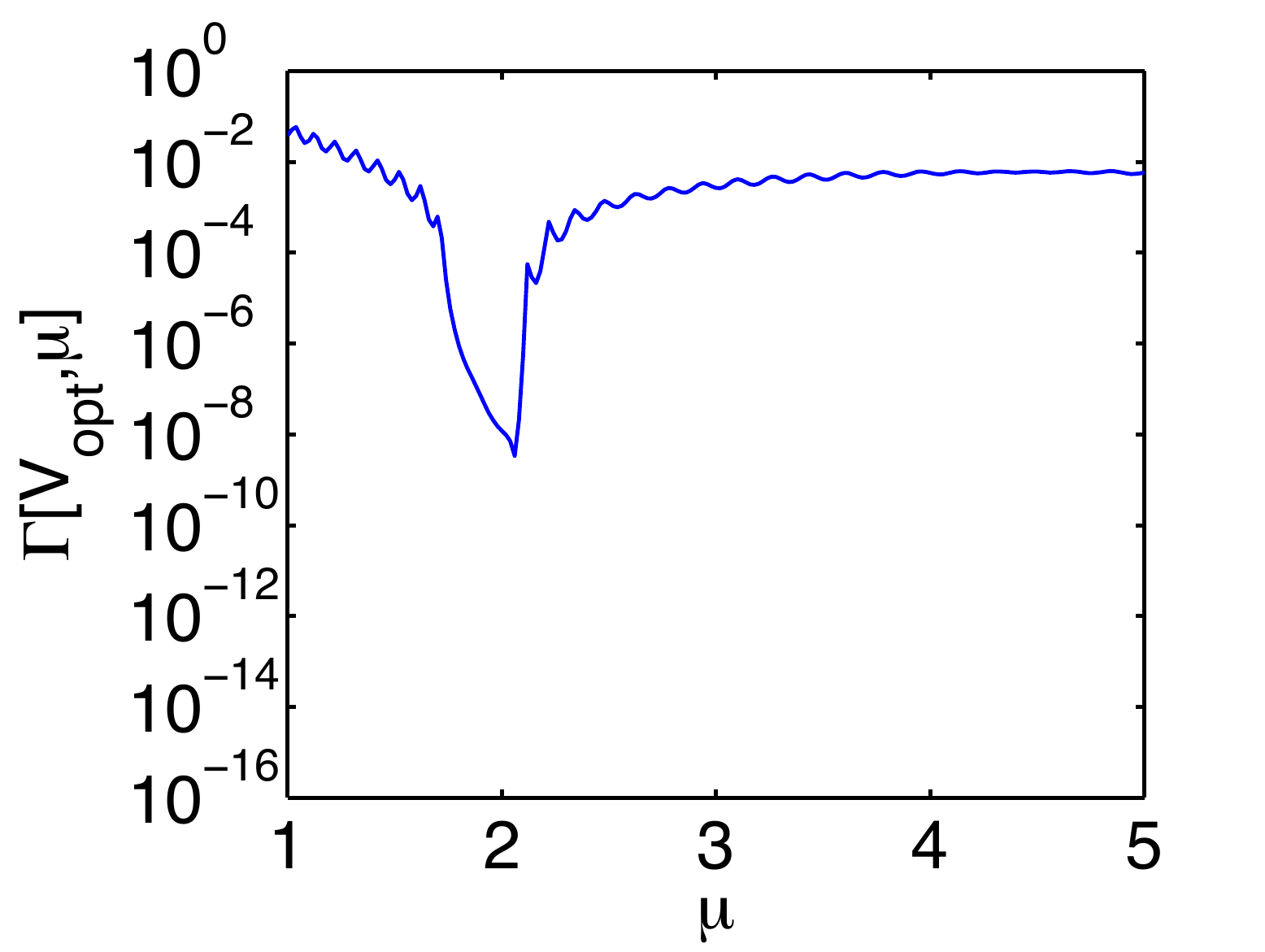} &
\includegraphics[width=2in]{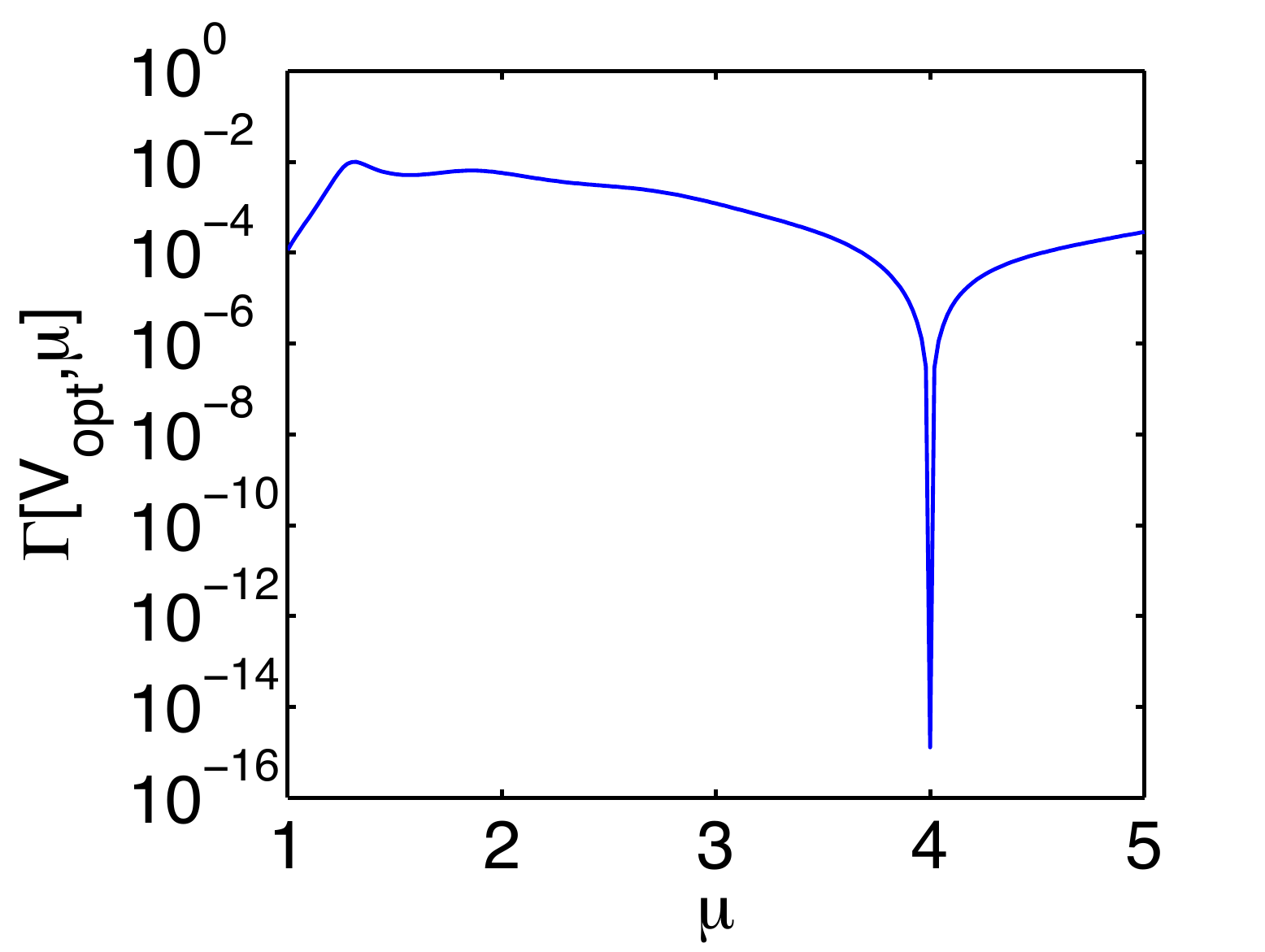}
\end{tabular}
\caption{Comparison of two locally optimal potentials achieving small $\Gamma[V]$  due the low density of states mechanism (left) and the cancellation mechanism (right). 
First row: displays plots of potentials. 
Second row: transmission, $t_V(k)$, a measure of the density of states. At the resonant frequency, $k_V$, indicated by the arrow, $t_V(k)$ is very small on the left and  approximately one on the right. The left figure inset shows that the resonant frequency is distinct from the resonant spike in the ``gap''. 
Third row: distorted plane waves,  $|e_{V+}(x,k_V)|$. 
Fourth row: $\Gamma[V_{opt};\mu]$ vs. $\mu$  for $V=V_{opt,A}$ optimized for forcing frequency $\mu_A=2$ (left),  and for $V=V_{opt,B}$ optimized for forcing frequency $\mu_B=4$ (right). Note contrasting sensitivity to perturbations in  $\mu$ away from $\mu_{A,B}$.  }
\label{fig:2mech}
\end{figure}


The functional to be minimized, $\Gamma[V]$, is given by
\begin{equation}
\Gamma[V] = 
\frac{1}{16\ k_V}\ |t_V(k_V)|^2\ \sum_\pm |\langle \beta \psi_V, f_{V\pm}(\cdot,k_{V}) \rangle  |^2;
\nn\end{equation}
see \eqref{eqn:Gamma}. 
 As discussed in the introduction, two possible mechanisms can be used to decrease the values of $\Gamma[V]$;
 see Remark \ref{rem:minGamma}:
 
  \noindent {\bf Mechanism (A)} Find a potential in $\mathcal{A}_1$ for which the first factor in \eqref{toy-fgr}, $|t_V(k_{res})|^2$ is small, corresponding to low {\it density of states} near $k_{res}^2$,  or
 
 \noindent {\bf Mechanism (B)} Find a potential in $\mathcal{A}_1$ which may have significant density of states near $k_{res}^2$  (say $|t_V(k_{res})|\ge 1/2$) but such that the oscillations of $f_V(x,k_{res})$ are tuned to make the matrix element expression (inner product) in \eqref{toy-fgr} small due to cancellation in the integral.
\medskip

In Fig. \ref{fig:2mech}
%
%
 we display the results of numerical simulations illustrating examples of both mechanisms at work.
 On the left is the potential, $V_{A,opt}(x)$, and diagnostics exhibiting mechanism (A) and on the right we exhibit mechanism (B) for the potential labeled $V_{B,opt}(x)$.  
 For both examples we choose $\beta(x) = \mathbf{1}_{[-2,2]}(x)$.\medskip
 
 The potential $V_{A,opt}(x)$ is obtained via optimization on the set $\mathcal{A}_1^\reg$ with $a=64$ and $\mu_A=2$ (same as in Fig. \ref{fig:compareParams} (left)). 
  The potential $V_{B,opt}(x)$ is obtained via optimization on the set $\mathcal{A}_1^\reg$ with $a=8$ and $\mu_B=4$. 
The first row of  figures displays the potentials as functions of $x$.
The value of $\Gamma$ for $V_{A,opt}(x)$ and $V_{B,opt}(x)$ are $\Gamma\left[V_{A,opt}\right] = 1.2 \times 10^{-8}$ and $\Gamma\left[V_{B,opt}\right] = 1.3 \times 10^{-15}$. 

The second row of plots is of the transmission coefficients $|t_V(k)|^2$ (see Eq. (\ref{eqn:defTransRefl}))  of $V_{A,opt}(x)$ and $V_{B,opt}(x)$. 
The small vertical arrows along the $k-$ axes  indicate the location of the resonant frequency $k=k_V$.

\begin{rem}  
\label{rem:lowDOS}
 {\bf Relevance of the transmission coefficient, $t_V(k)$, to the density of states:} Consider a periodic potential, $q(x)$, defined on $\mathbb{R}$. The spectrum of $-\partial_x^2+q(x)$ is equal to the union of closed intervals (bands) separated by open intervals (gaps). Now consider $q_M(x)=q(x)\mathbf{1}_{[-M,M]}(x)$. The decaying potential $q_M(x)$ has continuous spectrum extending from  zero to infinity. We expect however the spectral measure, associated with the self-adjoint operator, $-\partial_x^2+q_M(x)$ for $M\gg1$, to have little mass on those intervals corresponding to the gaps 
in the spectrum of the limit operator, $-\partial_x^2+q(x)$. Related to this is the observation that the $t_{q\ \mathbf{1}_{[-M,M]}}(k)$,  for $-\partial_x^2+q(x)\mathbf{1}_{[-M,M]}(x)$, is uniformly small, for $k^2$ in the spectral gaps of the limit operator, and converge weakly to one for $k^2$ in the spectral bands; see, for example,
 \cite{Barra:1999jt,Iantchenko:2006kl}. Thus, by plotting the amplitude of the transmission coefficient for our optimal potentials
  we can anticipate whether the density of states is small and a spectral gap is being opened around the resonant frequency, $k_V$. Thus, if $k_V$ lies in an interval of very low transmission, $t_V(k)$, the $\Gamma$, given by \eqref{eqn:Gamma} will be small. 
\end{rem}
The left plot in the second row shows that the transmission coefficient for $V_{A,opt}(x)$
 is very close to zero very near the resonant frequency, $k_{V_{A,opt}}$. 
 On the right we see that for $V_{B,opt}(x)$ the transmission coefficient very near  $k_{V_{A,opt}}$  close to one. 
 \medskip
 
In the third row of plots, for each potential, we plot the modulus of the  distorted plane wave at the resonant frequency, $|e_{V+}(x,k_V)|$. 
(Recall that for a symmetric potential, $e_{V-}(x,k) = e_{V+}(-x,k)$.)  The modulus of the distorted plane wave associated with $V_{A,opt}(x)$ decays rapidly as it enters the support of the potential, as expected since the transmission coefficient is nearly zero for this frequency (see Eq. (\ref{eqn:defTransRefl})).  The modulus of the distorted plane wave associated with $V_{B,opt}(x)$ is nearly unity over the support of the potential. 
\medskip

In the bottom row of plots of Fig. \ref{fig:2mech} we highlight an additional distinction between these two mechanisms. 
We fix the optimal potentials, $V_{A,opt}(x)$ and $V_{B,opt}(x)$,  respectively optimized for forcing fixed frequencies $\mu_A$ and $\mu_B$.  We then consider the variation  of the function $\mu\mapsto\Gamma[V_{opt};\mu]$, where
\begin{equation}
\Gamma[V_{opt};\mu] \equiv 
\frac{1}{16 \sqrt{\lambda_{V_{opt}}+\mu}}\ \left| t_{V_{opt}}\left( \sqrt{\lambda_{V_{opt}}+\mu}\ \right) \right|^2\
 \sum_\pm \left| \left\langle \beta \psi_{V_{opt}}, f_{V_{opt}\pm} \left(\cdot,\sqrt{\lambda_{V_{opt}}+\mu}\ \right)
                  \right\rangle  \right|^2.
\nn\end{equation}
Here, $\mu$ varies over a range of forcing frequencies above and below $\mu_A$, respectively, $\mu_B$.

We find that for $V_{A,opt}(x)$, the value of $\Gamma$ is relatively insensitive  to small changes in $\mu$ near $\mu_A$. 
Indeed, this is expected. Small variations in  $\mu$, imply small variations in  $\sqrt{\lambda_V + \mu}$. Therefore, if $k_{V_{A,opt}}=\sqrt{V_{A,opt}+\mu_A}$ is located in a spectral ``gap", then for  values of  $\mu$ near $\mu_A$, 
 $k(\mu)\equiv\sqrt{V_{A,opt}+\mu}$ is also in this ``gap'' . Therefore, $t_{V_{A,opt}}(k(\mu))$ and therefore $\Gamma[V_{A,opt},\mu]$ is small.  
 
In contrast, for $V_{B,opt}(x)$, the range of $\mu$ for which $\Gamma[V_{B,opt};\mu]$ remains small is extremely narrow; the smallness of the oscillatory integral, $\Gamma[V_{B,opt};\mu]$, is not preserved over a range of values of $\mu$. 

\begin{rem} These observations on the  sensitivity of $\Gamma[V_{opt},\mu]$ with respect to the forcing frequency, $\mu$, for the two different kinds of optimizers, $A-$ type  and $B-$ type, should have ramifications for applications.
\end{rem}
\medskip

By Proposition \ref{prop:tkGreaterZero},
\begin{equation}
V\in \cA_1^\reg(a,b,\mu) \ \implies\ \  |t_V(k)|\ \ge \exp\left(- 4a^2b\right).
\label{tVbound}
\end{equation} 
Thus we find that $\Gamma > 0$ due to Mechanism (A). 
However, in principle, one could find a potential such that due to perfect cancellation, $\Gamma = 0$ by mechanism (B). 
Indeed, the potential $V_B$ has an  extremely small value of $\Gamma$.

\subsection{Further discussion of mechanism (A); potentials which open a gap in the spectrum}
We have observed  that some locally optimal potentials, {\it e.g.} the potential associated with the left column of Fig. \ref{fig:2mech}, have small values of $\Gamma$ due to mechanism (A), creating a low density of states at the resonant frequency $k_{V_{opt}}$. We explore this phenomena further here and discuss the relation to Bragg resonance. 

\begin{figure}[t!]
\centering
\begin{tabular}{cc}
\includegraphics[height=2in]{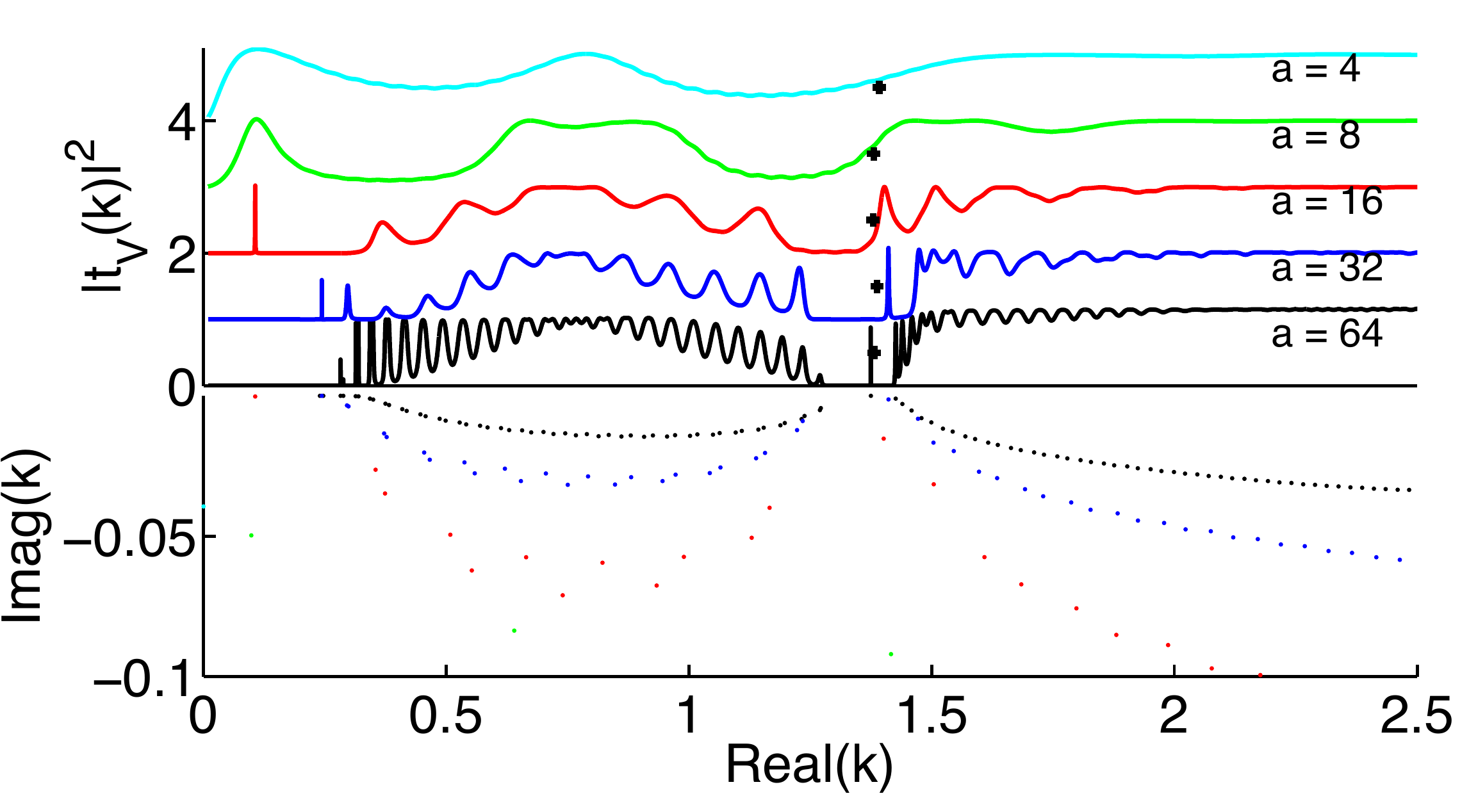} &
\includegraphics[height=2in]{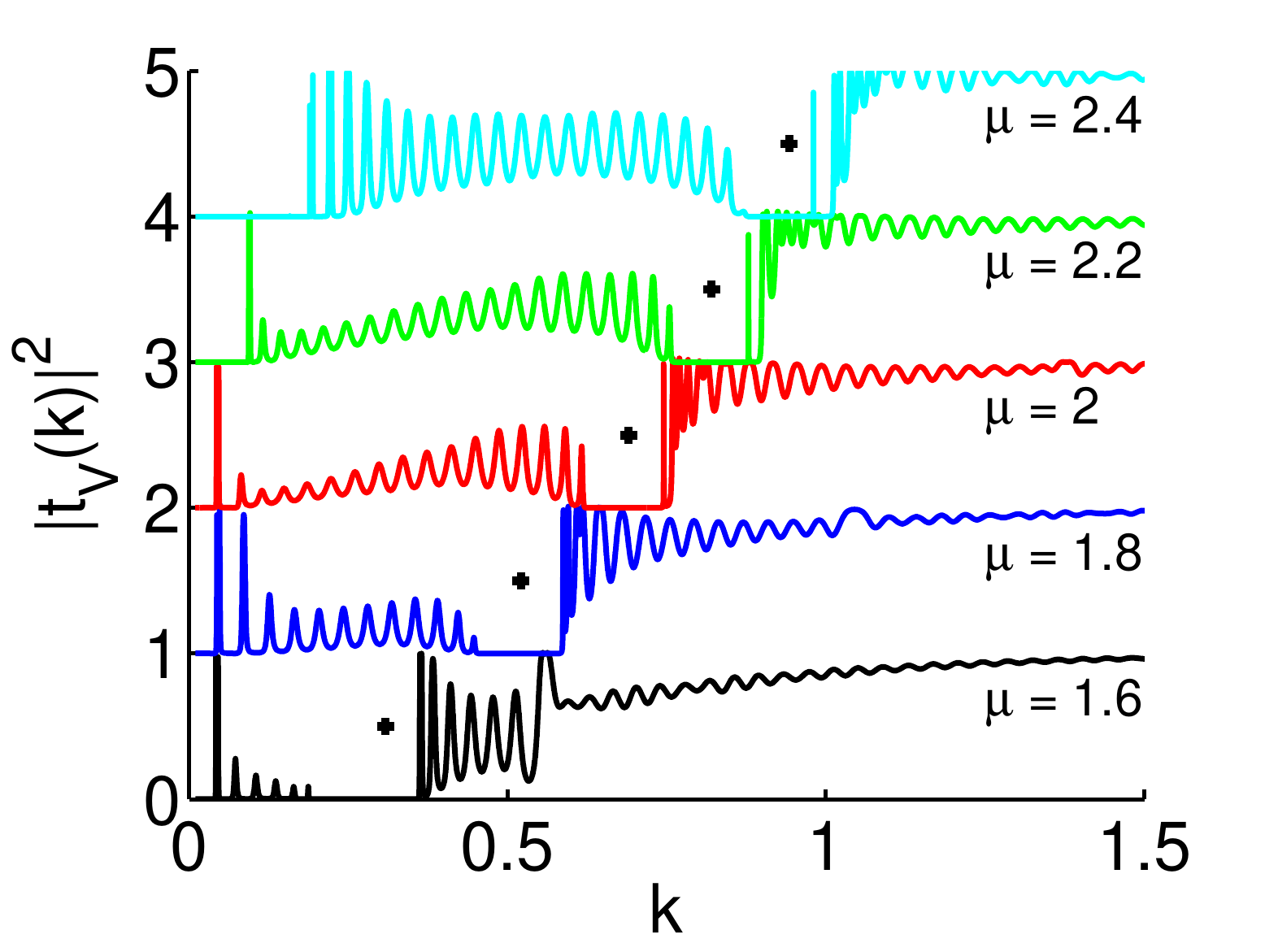} 
\end{tabular}
\caption{(Left) For the sequence of potentials in Fig. \ref{fig:compareParams}(left),  a spectral gap forms as $a\uparrow \infty$. (Right) For each of the potentials in Fig. \ref{fig:compareParams}(right), the resonant frequency lies in a spectral gap. }
\label{fig:compareParams2}
\end{figure}


For the sequence of potentials given in Fig. \ref{fig:compareParams} (left) corresponding to an increasing sequence of values for the support parameter $a$, we plot in Fig. \ref{fig:compareParams2} (left) the transmission coefficients (top) and resonances in the lower complex plane (bottom) in corresponding colors. For each potential, the location of the resonant frequency, $k_{V_{opt}}$ is indicated by a black cross (\verb-+-) in the transmission diagram. 
The resonances were computed by solving the associated quadratic eigenvalue problem using MatScat \cite{BindelWeb}. 

\begin{rem} 
As in Remark \ref{rem:lowDOS}, let $q(x)$ be a periodic potential and $q_M(x)=q(x)\mathbf{1}_{[-M,M]}(x)$. As $M\uparrow\infty$, the resonances of $-\partial_x^2 + q_M$ will converge to the spectrum of $-\partial_x^2 + q_\infty$  \cite{Barra:1999jt,Iantchenko:2006kl}.
\end{rem}

In Fig. \ref{fig:compareParams2} (left), we see from both the transmission coefficient and the resonances that a gap is opening in the spectrum as $a\uparrow \infty$, supporting Conjecture \ref{conj:convPer}, that $q_\infty(x)$ is a periodic potential with a localized defect. 

\begin{rem}
For large support parameter $a$, a narrow spike forms in the transmission coefficient for a value $k$ within the spectral gap of the limiting operator and a resonance lies nearby. 
In the limit that  $a\uparrow \infty$, this resonance converges to a point eigenvalue within the spectral gap. 
For periodic potentials with a localized defects, such defect eigenvalues exist \cite{Hoefer:2010fk,Figotin-Klein:97,Figotin-Klein:98,Johnson-etal:10}. Our $V_{opt}$ are qualitatively similar to the class studied in  \cite{Hoefer:2010fk}.
Note that the spike in the transmission coefficient in Fig. \ref{fig:compareParams2} (left) appears to lie near the resonant frequency, but at a distinct value.
\end{rem}

In Fig. \ref{fig:compareParams2} (right) we plot $k$ vs. the transmission coefficient $|t_V(k)|^2$ for the color-corresponding potentials obtained by varying $\mu$ (the forcing frequency for which the optimization is performed) in Fig. \ref{fig:compareParams} (right). 
For each value of $\mu$, the resonant frequency lies in a spectral gap for each value of $\mu$ and there appears only to be a single gap. 

\begin{rem}
The Schr\"odinger operator $H_q = -\partial_x^2 +q$ with one specified spectral gap is unique and can be explicitly  written in terms of Jacobi elliptic functions \cite{Hochstadt:1965uq}. These are called {\it one-gap potentials}. Using the transmission coefficient plots in Fig \ref{fig:compareParams2} (right) to estimate the position of the spectral gap, we find that the corresponding one-gap potential has period comparable to that of the periodic tail of the potentials given in Fig. \ref{fig:compareParams}(right). 
\end{rem}

This suggests a good heuristic for finding potentials with small values of $\Gamma[V]$: Start with a localized potential well supporting a single bound state. Then, create a low density of states at $k = \sqrt{\lambda_V + \mu}$ by adding a truncated one-gap potential  with appropriate spectral gap. If the potential added has small amplitude, then this heuristic is equivalent to adding a cosine or Mathieu potential with  frequency given by the Bragg relation.

\subsection{Optimizing $\Gamma$ with $\beta = V$ as in Eq. \eqref{betaVgam}}
Here we study the case where $\beta=V$ as in Remark \ref{rem:2classes}, Eq. \eqref{betaVgam}, and Corollary \ref{cor:betaEqualV}. 

In Fig \ref{fig:betaEqualV} (left), we take $\mu = 2$ and plot locally optimal potentials for 4 different values of the support, $a$. The values of $\Gamma$ are given in the following table.
\begin{center}
\begin{tabular}{c| l l l l}
$a$ & 4 & 8 & 16 & 32  \\
\hline
$\Gamma$ & $8 \times 10^{-13}$ & $3\times 10^{-13}$ & $2\times 10^{-13}$ & $2\times 10^{-12}$  
\end{tabular}
\end{center}

In Fig. \ref{fig:betaEqualV} (right), we take $a=32$, and plot locally optimal potentials for 4 values of forcing frequency $\mu$. The values of $\Gamma$ are given in the following table.

\begin{center}
\begin{tabular}{c| l l l l l}
$\mu$ &  2 & 3 & 4 & 5 \\
\hline
$\Gamma$ & $2\times 10^{-12}$  & $1\times 10^{-12}$ & $2\times 10^{-13}$ &$2\times 10^{-15}$  
\end{tabular}
\end{center}

As noted in Remark \ref{lem:nonConvex}, the solution of the potential design problem is not guaranteed to be unique,  since the admissible set is non-convex. Regarding Conjecture \ref{conj:convPer} on the character of the limit of optimizers, $V_{opt,a}$ as $a$ tends to infinity, since for $\beta(x)=V(x)$ and $a=\infty$, the functional $V\mapsto\Gamma[V]$ is invariant under the transformation $V(x)\mapsto V(x+x_0)$, we could expect convergence to $V_{opt,\infty}(x)$, a localized perturbation of a periodic potential, only modulo translations. 

\begin{figure}[t!]
\centering
\centering
\begin{tabular}{cc}
\includegraphics[width=2.5in]{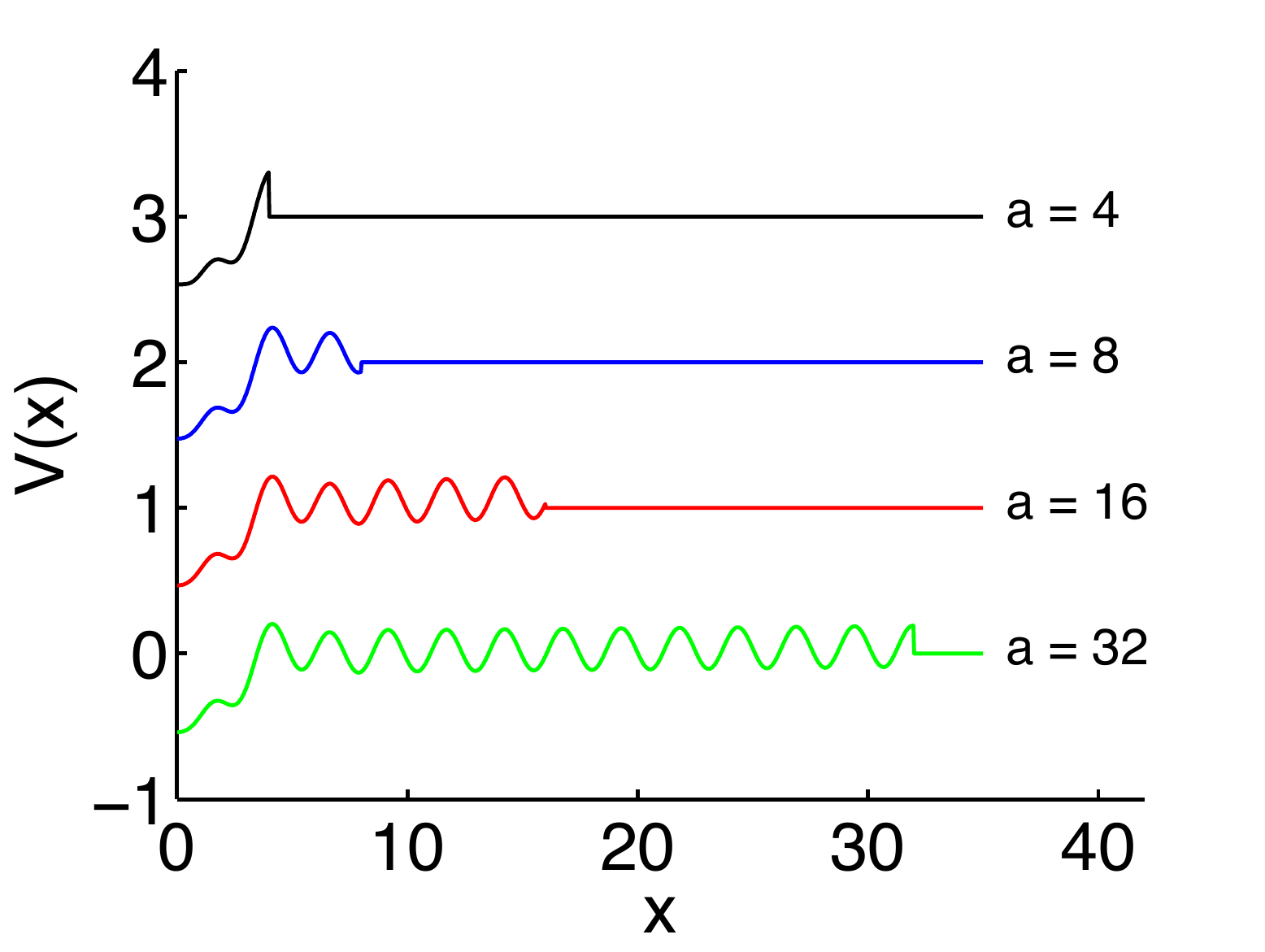}
\includegraphics[width=2.5in]{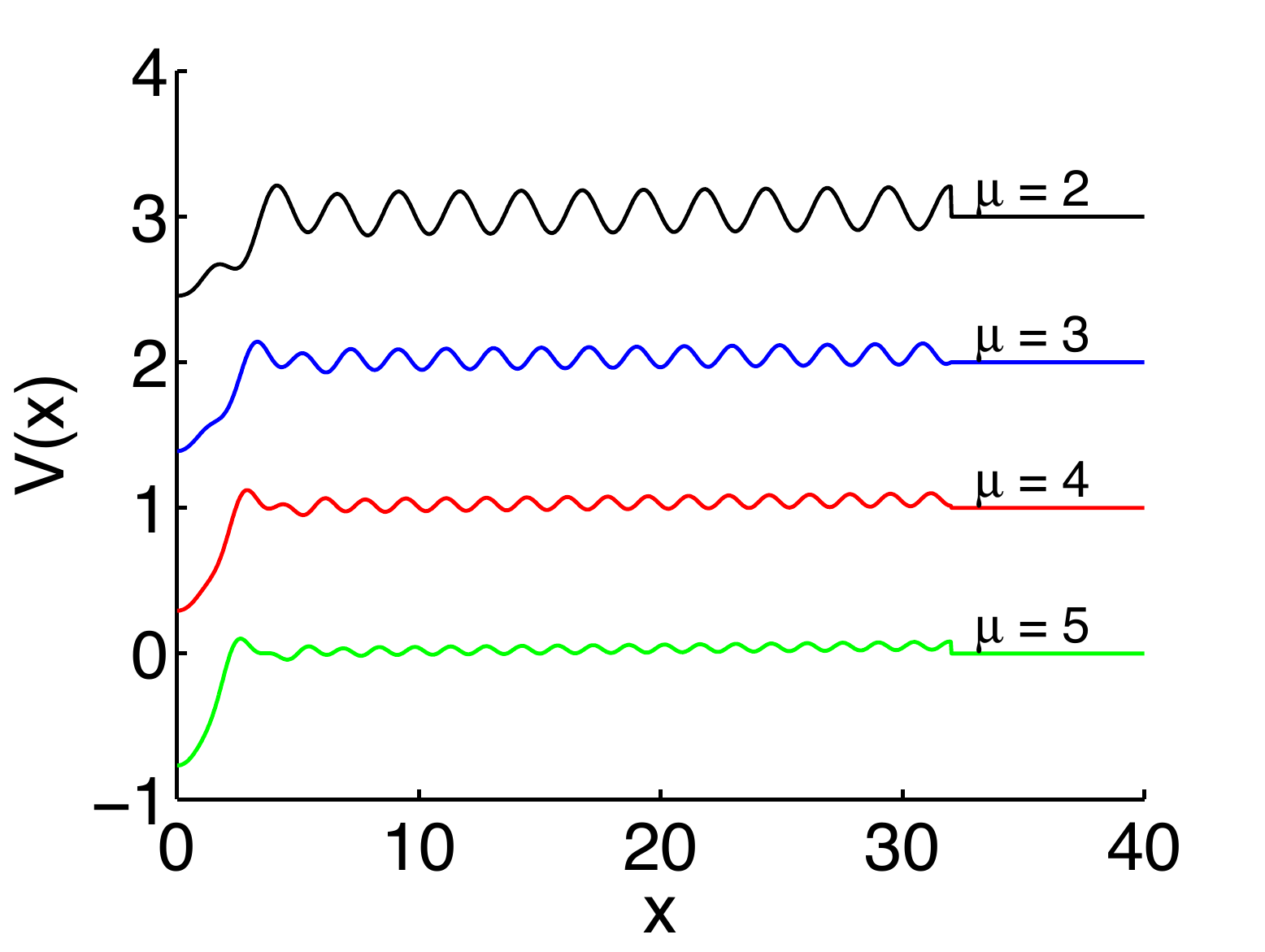} \\
\end{tabular}
\caption{With $\beta = V$ as in Eq. \eqref{betaVgam}, we plot locally optimal potentials for varying values of support $a$ and forcing frequency $\mu$. }
\label{fig:betaEqualV}
\end{figure}


\subsection{Time dependent simulations}
\label{timeDependentSim}
For a locally optimal potential of the potential design problem,  \eqref{eqn:PDP}, we independently verify that the potential supports a  very long-lived metastable state by conducting time-dependent simulations. 
See Section \ref{sec:numMeth} for a discussion of the numerical method. 
We set $V_{init} = -A \sech^2(Bx)$ for suitably chosen $A$, $B$ and take $V_{opt}$ to be a locally optimal solution to the PDP \eqref{eqn:PDP} with $\beta = \mathbf{1}_{[-2,2]}$, $\mu = 2$, and $a=12$ (same parameter choice as in Section \ref{sec:increaseSupport}). We then solve the parametrically forced Schr\"odinger Eq.  (\ref{forced-schrod}) with $\epsilon =1$ until $t=40$ with initial conditions given by  the ground state of $H_V$ for the two potentials, {\it i.e.} $\phi^\epsilon(0) = \psi_{V_{opt}}$ and $\phi^\epsilon(0) = \psi_{V_{init}}$. 
In Figs. \ref{fig:nonAutOpt} (left) and \ref{fig:nonAutOpt} (center) we plot  $V$, $\beta$, and $\psi_V$ for the two potentials. In Figure \ref{fig:nonAutOpt} (right), we plot $t$ vs. $|\langle \phi^\epsilon(t,\cdot), \psi_V(\cdot) \rangle |^2$, the square modulus of the projection of the wave function onto the bound state for the two potentials.  

\subsection{Filtering study}
\label{sec:Filtering}
For the same potentials studied in Sec. \ref{timeDependentSim} and Fig. \ref{fig:nonAutOpt} plus the one studied in Fig. \ref{fig:2mech}(right), we conduct the following experiment. We consider the time evolution of Eq. (\ref{forced-schrod}) until time $t=50$ with initial condition taken to be $\psi_V + \text{noise}$. The noise is taken to be normally distributed random numbers generated using Matlab's \verb+randn+ function for each point in the interval $[-a,a]$. The initial condition is then normalized so that $\langle \phi^\epsilon(0) ,\, \psi_V\rangle = 1$. 
The results are plotted in Fig. \ref{fig:filtering}. We find  that for a non-optimized potential, the final state of the system is nearly zero. While for the locally optimal potential, the  bound state emerges as the  final state. In the central panel of 
Figure \ref{fig:filtering}  we see convergence to the projection of the initial condition onto the bound state of $H_{V_A}$; see central panel of Figure \ref{fig:nonAutOpt}. 
In the right panel of 
Figure \ref{fig:filtering}  we see convergence to the projection of the initial condition onto the bound state of $H_{V_B}$. 

This study suggests that such a device could be used as a filter to select  a particular spatial mode profile.  For these potentials, the system behaves as a mode-selecting waveguide, preserving the discrete components of the initial condition, while radiating the continuous components. 
 Alternatively, this study demonstrates the robustness of $\psi_{V_{opt}}$ to large fluctuations in the data.

\begin{figure}[t!]
\centering
\centering
\begin{tabular}{ccc}
\includegraphics[width=2.0in]{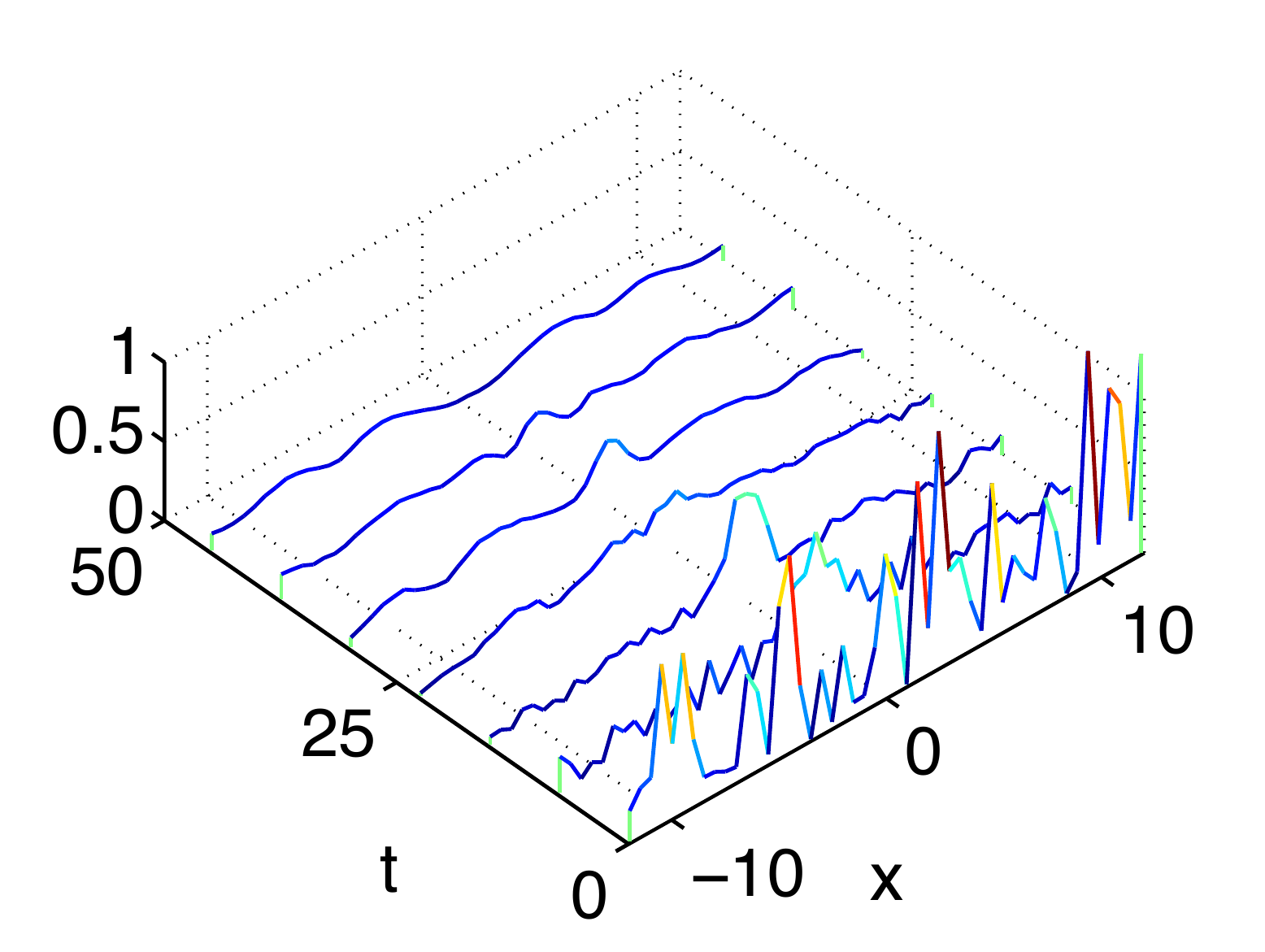} &
\includegraphics[width=2.0in]{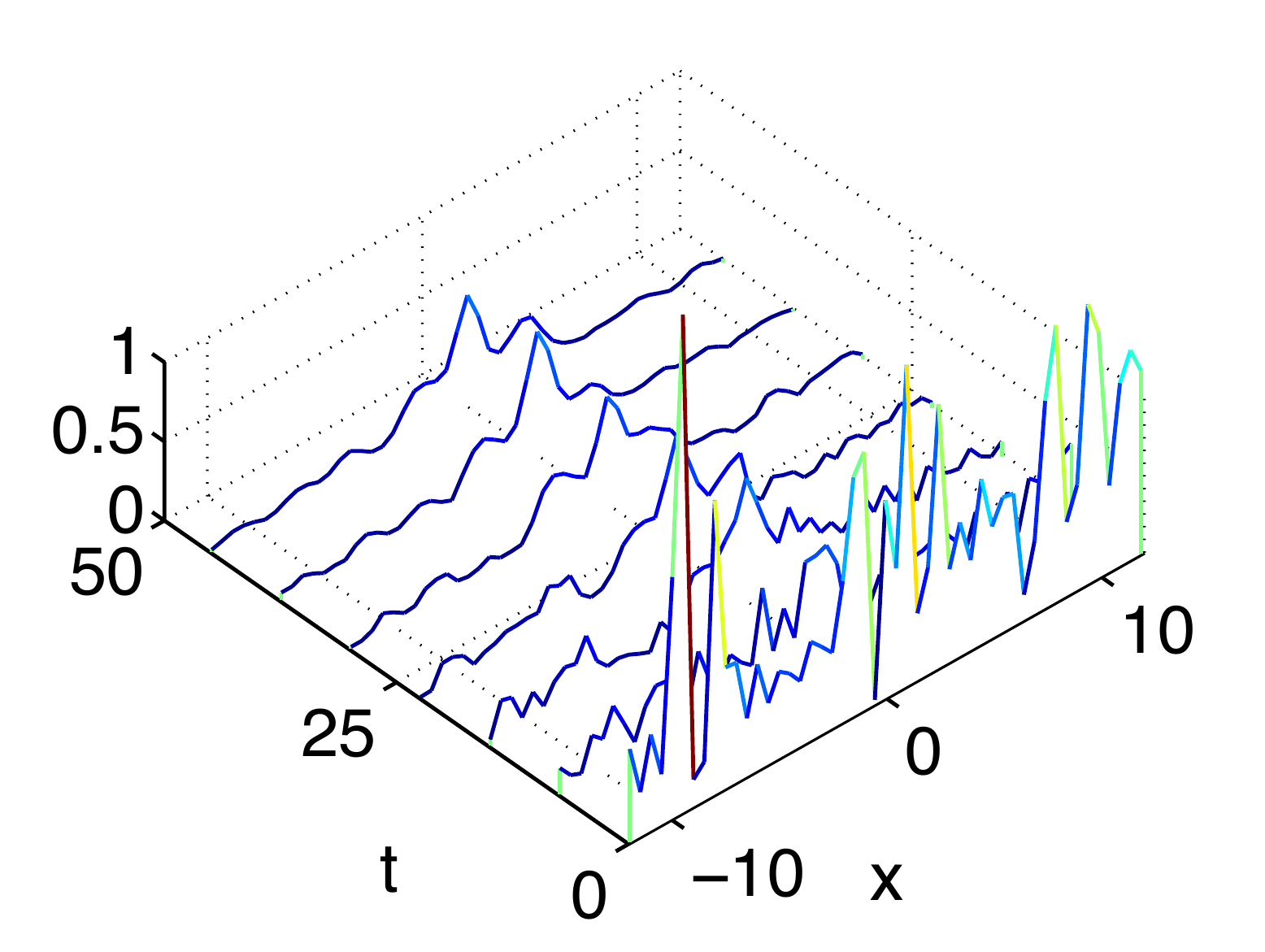} &
\includegraphics[width=2.0in]{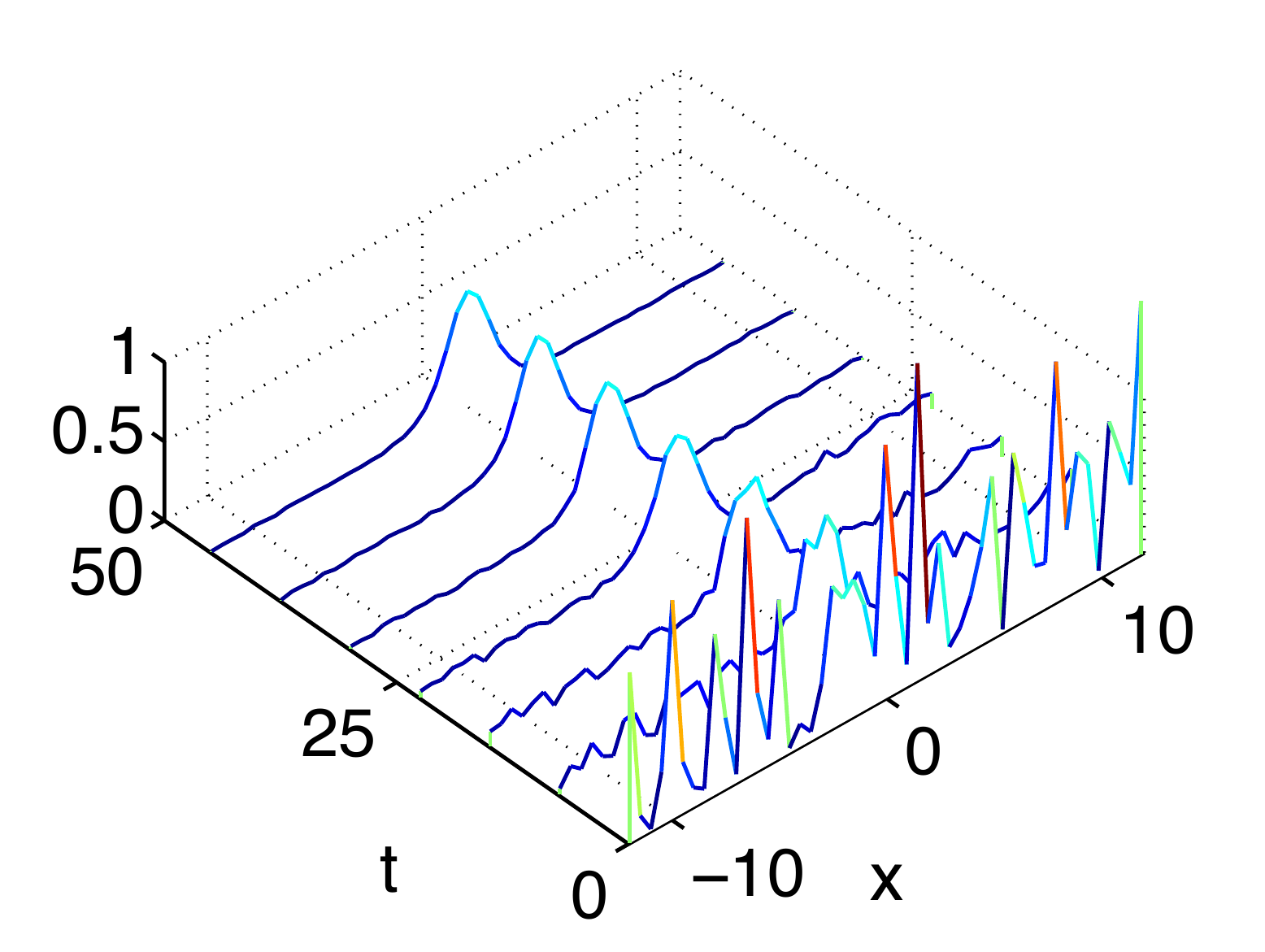}
\end{tabular}
\caption{For the two potentials in Sec. \ref{timeDependentSim} and Fig. \ref{fig:nonAutOpt}, and cancellation potential in Fig. \ref{fig:2mech}(right), we plot the time evolution $\phi^\epsilon(t,x)$, governed by Eq. (\ref{forced-schrod}) with initial condition taken to be $\psi_V + \text{noise}$.  The simulation was performed on a spatial domain $[-60,60]$ with absorbing boundary conditions. }
\label{fig:filtering}
\end{figure}

\section{Discussion and conclusions}
\label{sec:Discussion}
Scattering loss, a conservative loss mechanism,  is often a limiting factor in the performance of many engineered devices. Therefore, there is great interest in finding structures with  low scattering loss-rate. Loss can occur due to parametric or nonlinear time-dependent perturbations which couple an ideally isolated state to an environment.  
We consider a model of a bound state supported by a potential, $V$, subject to a time-periodic and spatially localized  ``ionizing'' perturbation. 
 The rate of scattering loss, $\Gamma[V]$, due to coupling of the bound state to radiation modes is given by Fermi's Golden Rule, which depends on the potential $V$. Using gradient-based optimization methods we find locally optimal structures with much longer-lived bound states.  These potentials appear to be truncations of smooth periodic structures with localized defects.
This approach can be extended to the wide class of problems presented in the introduction.

\appendix

\section{Computation of gradients / functional derivatives}
\label{sec:appCompGrads}

\subsection{Proof of Prop. \ref{eqn:grad}, gradient of $\Gamma[V]$}
\begin{proof}
Here we compute the individual terms in Eq. (\ref{eqn:chainRuleGamma}) and then assemble below. 

\paragraph{Computation of $\frac{\delta \Gamma}{\delta \psi_V}$ and $\frac{\delta \Gamma}{\delta e_{V\pm}}$.} 
We compute
\begin{subequations}
\label{eqn:dGdpsi}
\begin{align}
\frac{\delta \Gamma}{\delta \psi_V}[\delta \psi] &= \frac{1}{16 k_V} \sum_\pm \langle \beta \delta \psi, e_{V\pm}\rangle \overline{\langle \beta \psi_V, e_{V\pm}\rangle} + \text{c.c.} \\
&= \frac{1}{8 k_V} \Re \sum_\pm \langle \beta \psi_V, e_{V\pm}\rangle \langle \beta e_{V\pm}, \delta \psi \rangle. 
\end{align}
\end{subequations}
Similarly,
\begin{align}
\label{eqn:dGdeV}
\frac{\delta \Gamma}{\delta e_{V\pm}} [\delta e_V]&= \frac{1}{8 k_V} \Re \langle e_{V\pm}, \beta \psi_V  \rangle \langle \beta \psi_V, \delta e_V\rangle.
\end{align}

\paragraph{Computation of $\frac{\delta \lambda_V}{\delta V}$ and $\frac{\delta \psi_V}{\delta V}$. }
Taking variations of $H_V \psi_V = \lambda_V \psi_V$ we find that 
$$
(H_V - \lambda_V) \delta \psi_V = - (\delta V \psi_V - \delta \lambda_V \psi_V).
$$
Multiplying by $\psi$ and integrating, we obtain $\delta \lambda_V = \langle \psi_V, \delta V \psi_V \rangle$, {\it i.e.} 
\begin{align}
\label{eqn:dlamdV}
\frac{\delta \lambda_V}{\delta V} = \psi_V^2
\end{align}
and 
\begin{align*}
(H_V - \lambda_V) \delta \psi_V &= - (\delta V \psi_V - \langle \psi_V, \delta V \psi_V \rangle \psi_V ) \\
&\equiv - P_{\lambda_V}^\perp \left[\psi_V  \delta V\right] .
\end{align*}
where $P_{\lambda_V}^\perp$ is the orthogonal projection onto the space spanned by $\psi_V$.  Since $H_V$ supports only a single bound state, $P_{\lambda_V}^\perp = P_c$. The solution of this equation can be written in terms of the resolvent operator
\begin{align}
\label{eqn:dpsidV}
\frac{\delta \psi_V}{ \delta V} [\delta V] = \delta \psi_V = - R_V(\sqrt{\lambda_V}) P_c [\psi_V \delta V]
\end{align}

\paragraph{Computation of $\frac{\delta e_{V\pm}}{\delta V}$ and $\frac{\delta k_V}{\delta V}$.} We note that the distorted plane waves $e_{V\pm}$ can be expressed
\begin{equation}
\label{eqn:eVinTermsOfPhi}
e_{V\pm}(x,k_V) = e^{\pm \imath k_V x} - \phi_{V\pm} (x,k_V)
\end{equation} 
where $\phi_{V\pm}$ satisfies the following equation with outgoing boundary conditions
\begin{subequations}
\label{eqn:phi}
\begin{alignat}{2}
(H_V - k_V^2) \phi_{V\pm} &= V e^{\pm \imath k_V x} &\qquad& x \in \Omega=[-L,L] \\
\nabla\phi_{V\pm} \cdot \hat{\mathbf n} &= \imath k_V \phi_{V\pm} &\qquad& x \in \partial \Omega. 
\end{alignat}
\end{subequations}
Taking variations of Eq. (\ref{eqn:phi}), we obtain
\begin{subequations}
\label{eqn:DEdeltaPhi}
\begin{align}
(H_V - k_V^2) \delta \phi_{V\pm} &= \delta V e_{V\pm} +   
\left( 2k_V \phi_{V\pm} \pm \imath x V e^{\pm \imath k_V x} \right) \delta k_V \\ 
\nabla \delta \phi_{V\pm} \cdot \hat {\mathbf n} - \imath k_V \delta \phi_{V\pm} &= \imath \delta k_V \phi_{V\pm}.
\end{align}
\end{subequations}
Recalling $k_V^2 = \lambda_V + \mu$ and Eq. (\ref{eqn:dlamdV}) we find that $\delta k_V = \frac{\delta k_V}{\delta V}[\delta V] =  \langle \frac{\psi_V^2}{2k_V}, \delta V \rangle $ or equivalently 
\begin{align}
\label{eqn:dkdV}
\frac{\delta k_V}{\delta V} =  \frac{\psi_V^2}{2k_V}. 
\end{align}
Equation (\ref{eqn:DEdeltaPhi}) is a forced equation for $\delta \phi_{V\pm}$ with a unique solution since there is no nontrivial, outgoing solution to the homogenous equation \cite{tang-zworski-1dscat}.  
The general solution of Eq. (\ref{eqn:DEdeltaPhi}a) is
$$
\frac{\delta \phi_{V\pm}}{\delta V} [\delta V] = \alpha_\pm e_{V+} + \beta_\pm e_{V-} + R_V(k_V)\left[ \delta V e_{V\pm} +\left( 2k_V \phi_{V\pm} \pm \imath x V e^{\pm \imath k_V x} \right) \delta k_V \right]
$$ 
where $\alpha$, $\beta$ are constants. Matching boundary conditions in Eq. (\ref{eqn:DEdeltaPhi}b) and recalling that $R_V$ is the outgoing resolvent, we obtain 
\begin{align*}
\alpha_\pm &= - \frac{\delta k_V}{2k_V} \phi_{V\pm}(-L) e^{\imath k_VL} \\
\beta_\pm &= - \frac{\delta k_V}{2k_V} \phi_{V\pm}(L) e^{\imath k_VL}
\end{align*}
so that
\begin{align}
\nonumber
\frac{\delta \phi_{V\pm}}{\delta V}[\delta V]  =&  \left( R_V(k_V)\left[ 2k_V \phi_{V\pm} \pm \imath x V e^{\pm \imath k_V x} \right] -\frac{e^{\imath k_V L}}{2 k_V}\left( \phi_\pm(-L) e_{V+} + \phi_\pm(L) e_{V-} \right)   \right)  \frac{\delta k_V}{\delta V}[\delta V] \\
& + R_V(k_V) [ \delta V e_{V\pm}]. 
\label{eqn:deltaPhi}
\end{align}
Now using Eq. (\ref{eqn:eVinTermsOfPhi}) we find that
\begin{equation}
\label{eqn:deVdV2}
\delta e_{V\pm} = \pm \imath x \delta k_V e^{\pm \imath k_V x} -  \delta \phi_{V\pm}.
\end{equation}

\paragraph{Computation of Terms in Eq. (\ref{eqn:chainRuleGamma}).}
Using Eqs. (\ref{eqn:dGdpsi}) and (\ref{eqn:dpsidV}) we obtain for the first term in Eq. (\ref{eqn:chainRuleGamma})
\begin{align}
\label{eqn:term1}
 \frac{\delta \Gamma}{\delta \psi_V} \left[ \frac{\delta \psi_V}{\delta V} [ \delta V ] \right] = 
\langle - \frac{1}{8k_V}  \psi_V R_V(\sqrt{\lambda_V}) P_c \left[ \Re \sum_\pm \langle e_{V\pm} , \beta \psi_V \rangle \beta e_{V\pm} \right] , \delta V \rangle
\end{align}
where we have used the fact that the operator $R_V(\sqrt{\lambda}) P_c \colon L^2 \rightarrow L^2$ is symmetric.

The second term of Eq. (\ref{eqn:chainRuleGamma}) can be written using Eqs. (\ref{eqn:dGdeV}), (\ref{eqn:dkdV}), (\ref{eqn:deltaPhi}), and (\ref{eqn:deVdV2})
\begin{subequations}
\label{eqn:term2}
\begin{align}
\frac{\delta \Gamma}{\delta e_{V\pm}} \left[   \frac{\delta e_{V\pm}(\cdot, k_V)}{\delta V} [\delta V] \right] 
&= \frac{1}{8k_V} \Re  \langle \beta \psi_V, e_{V\pm} \rangle  \langle \beta \psi_V,  A_\pm \delta k_V - R_V(k_V) \left[ \delta V e_{V\pm} \right] \rangle \\
\nonumber
&= \frac{1}{8k_V} \Re  \langle \beta \psi_V, e_{V\pm} \rangle  
\left(  \langle \beta \psi_V,  A_\pm \rangle  \langle \frac{\delta k_V}{\delta V}, \delta V \rangle -  \langle \overline{e_{V\pm}} R_V(k_V)\left[ \beta \psi_V\right] ,  \delta V  \rangle \right) \\
\nonumber
&= \langle \frac{1}{8k_V} \Re  \langle \beta \psi_V, e_{V\pm} \rangle   
\left(  \frac{1}{2k_V}\langle \beta \psi_V,  A_\pm \rangle \psi_V^2  - \overline{e_{V\pm}} R_V(k_V)\left[ \beta \psi_V\right] \right) , \delta  V  \rangle
\end{align}
\end{subequations}
where $A_\pm$ is given in Eq. (\ref{eqn:Apm}) and we have again used the fact that $R_V$ is a symmetric operator. 

Using Eq. (\ref{eqn:dkdV}), the third term of Eq. (\ref{eqn:chainRuleGamma}) is given by 
\begin{align}
\label{eqn:term3}
 - \frac{\Gamma}{k_V} \langle \frac{\delta k_V}{\delta V}, \delta V \rangle = - \frac{\Gamma}{2k_V^2} \langle \psi_V^2 , \delta V \rangle
\end{align}

From Eqs. (\ref{eqn:chainRuleGamma}), (\ref{eqn:term1}), (\ref{eqn:term2}), and (\ref{eqn:term3}) and the Riesz representation theorem we obtain Eq. (\ref{eqn:grad}) as desired. 
\end{proof}

\subsection{Proof of Prop. \ref{prop:wronDeriv}, gradient of $W_V(0)$ }
\begin{proof}
Denoting $\dot f(x) \equiv \frac{\delta f(x)}{\delta V}[\delta V(y)]$, we fix $x$ and compute
\begin{equation}
\label{eqn:dotW}
\dot W(x) =  \dot \eta_+ \eta_-'  +  \eta_+ \dot \eta_-' - \dot \eta_+' \eta_- - \eta_+' \dot \eta_-.
\end{equation}
To compute $\dot \eta_\pm$, we take variations of Eq. (\ref{eqn:halfBoundState}) to obtain 
\begin{align*}
H_V \delta \eta_\pm &=  - \delta V \eta_\pm  \\
\lim_{x\rightarrow \pm \infty} \partial_x  \delta\eta_\pm &= 0. 
\end{align*}
Using the variation of parameters formula, we find 
$$
\dot \eta_\pm(x) \equiv \frac{\delta \eta_\pm(x)}{\delta V}[\delta V] =  - \int q(x,y) \delta V(y)\eta_\pm(y) \ud y
$$
where
$$
q(x,y) = \frac{1}{W}
\begin{cases} 
\eta_-(x) \eta_+(y)  & x\leq y \\
\eta_+(x) \eta_-(y) & x\geq y
\end{cases}
$$
is the Green's function. Differentiating we find 
$$
\dot \eta_\pm'(x) =  - \int \partial_x q(x,y) \delta V(y)\eta_\pm(y) \ud y. 
$$
We now break Eq. (\ref{eqn:dotW}) into 2 parts: $\dot W = \dot W_1 + \dot W_2$ where $\dot W_1 = \int_{-\infty}^x \star \ud y$, $\dot W_2 = \int_x^\infty \star \ud y$, and the integrand is given by 
$$
\star = - \delta V(y) \left(q(x,y) \eta_+(y) \eta_-'(y) +  \eta_+(x) \partial_x q(x,y) \eta_-(y) - \partial_x q(x,y) \eta_+(y) \eta_-(x) + \eta_+'(x) q(x,y) \eta_-(y) 
\right) .
$$
We then evaluate 
\begin{align*}
\dot W_1 &= - \frac{1}{W}\int_{-\infty}^x \Big( 
 \eta_+(x) \eta_-(y) \eta_+(y) \eta_-'(x) + \underline{ \eta_+(x) \eta_+'(x) \eta_-(y) \eta_-(y)} \\
& \qquad - \eta_+'(x) \eta_-(y) \eta_+(y) \eta_-(y) - \underline{ \eta_+'(x) \eta_+(x) \eta_-(y) \eta_-(y)} 
\Big) \delta V(y) \ud y \\
&= - \int_{-\infty}^x \eta_+(y) \eta_- (y) \delta V(y) \ud y 
\end{align*}
where the underlined terms cancel and
\begin{align*}
\dot W_2 &= - \frac{1}{W}\int_x^\infty \Big( 
\underline{\eta_-(x) \eta_+(y) \eta_+(y) \eta_-'(x)} + \eta_+(x) \eta_-'(x) \eta_+(y) \eta_-(y) \\
& \qquad - \underline{\eta_-'(x) \eta_+(y) \eta_+(y) \eta_-(x)}  - \eta_+'(x) \eta_-(x) \eta_+(y) \eta_-(y)
\Big) \delta V(y) \ud y \\
&= - \int_x^\infty \eta_+(y) \eta_- (y) \delta V(y) \ud y.
\end{align*}
Thus 
$$
\dot W =  - \int  \eta_+(y) \eta_- (y) \delta V(y) \ud y
$$
and the result follows. Note that $\dot W(x,y)$ is constant in $x$ as expected.
\end{proof}



\small{
\bibliographystyle{amsplain}
\providecommand{\bysame}{\leavevmode\hbox to3em{\hrulefill}\thinspace}
\providecommand{\MR}{\relax\ifhmode\unskip\space\fi MR }
\providecommand{\MRhref}[2]{%
  \href{http://www.ams.org/mathscinet-getitem?mr=#1}{#2}
}
\providecommand{\href}[2]{#2}

}
\end{document}